\documentclass{article}

\usepackage{amsmath}
\usepackage{enumerate}
\usepackage{amsfonts}
\usepackage{amssymb}
\usepackage{graphicx}
\usepackage{mathrsfs}
\usepackage{esint}
\usepackage{upref,amsthm,amsxtra,exscale}
\usepackage{cite}
\usepackage[colorlinks=true,urlcolor=blue,
citecolor=red,linkcolor=blue,linktocpage,pdfpagelabels,
bookmarksnumbered,bookmarksopen]{hyperref}
\usepackage{cleveref}
\usepackage[cm]{fullpage}

\usepackage{subcaption}
\usepackage{caption}

\newtheorem{theorem}{Theorem}[section]
\newtheorem{corollary}[theorem]{Corollary}
\newtheorem{remark}[theorem]{Remark}
\newtheorem{lemma}[theorem]{Lemma}
\newtheorem{proposition}[theorem]{Proposition}

\numberwithin{equation}{section}

\def\N{\mathbb{N}}
\def\eps{\varepsilon}
\def\R{\mathbb{R}}
\def\S{\mathbb{S}^N}

\def\cE{\mathcal{E}}

\def\ast{*}
\def\weakto{\rightharpoonup}
\def\r{\mathbb{R}}

\def\rn{\mathbb{R}^N}

\def\eps{\varepsilon}

\def\H{\mathbb H(\S)}
\def\D{\mathbb D(\rn)}

\usepackage{marvosym}

\title{The conformal logarithmic Laplacian on the sphere:\\
Yamabe-type problems and Sobolev spaces}
\author{Juan Carlos Fernández\footnote{J.C. Fernández is supported by CONAHCYT grant CBF2023-2024-116 (Mexico) and  by UNAM-DGAPA-PAPIIT grant IN110225 (Mexico).}\,\, 
\&
Alberto Saldaña\footnote{A. Saldaña is supported by CONAHCYT grant CBF2023-2024-116 (Mexico) and by UNAM-DGAPA-PAPIIT grant IN102925 (Mexico).}\, 
\footnote{\Letter\, 
Corresponding author:
alberto.saldana@im.unam.mx
}}

\date{}

\begin{document}

\maketitle

\begin{abstract}

We study the conformal logarithmic Laplacian on the sphere, an explicit singular integral operator that arises as the derivative (with respect to the order parameter) of the conformal fractional Laplacian at zero. Our analysis provides a detailed investigation of its spectral properties, its conformal invariance, and the associated \(Q\)-curvature problem. Furthermore, we establish a precise connection between this operator on the sphere and the logarithmic Laplacian in \(\mathbb{R}^N\) via stereographic projection. This correspondence bridges classification results for two Yamabe-type problems previously studied in the literature, extending one of them to the weak setting. To this end, we introduce a Hilbert space 
that serves as the logarithmic counterpart of the homogeneous fractional Sobolev space, offering a natural functional framework for the variational study of logarithmic-type equations in unbounded domains.  Finally, we also include a derivation of the conformal logarithmic Laplacian on a cylinder. 

\medskip

(primary), 
35R11, 
58J40, 
58J70, 
58J90, 
(secondary).

\medskip


\end{abstract}

\section{Introduction}
In recent years, conformally covariant differential operators on Riemannian manifolds, along with their associated curvature quantities, have attracted significant attention. For the fractional-order case, see \cite{CG2011,GQ13,DM18,BettiolGonzalezMaalaoui2024,MayerNdiaye2024} and the references therein; for the integer-order case, see \cite{GJMS1992,GurskyMalchiodi2015,HangYang2016,MazumdarVetois2024,MazumdarNdiaye2026,MazumdarPremoselli2025} and the references therein. To be more precise, let \((M, h)\) be a Riemannian manifold of dimension \(N \in \mathbb{N}\), and let \(\mathscr{P}_h: \mathcal{C}^\infty(M) \rightarrow \mathcal{C}^\infty(M)\) denote a differential operator. We say that \(\mathscr{P}_h\) is conformally covariant if, under the change of metric $\eta h$, with $\eta\in\mathcal{C}^\infty(M)$ positive, it  satisfies a conformal transformation law of the form
\begin{equation}\label{VeryGeneralConformalLaw}
\mathscr{P}_{\eta h}(u) = \eta^{-b}\mathscr{P}_h(\eta^au),\qquad u\in\mathcal{C}^\infty(M),
\end{equation}
for some $a,b\in\mathbb{R}$. The associated notion of curvature given by $Q^{\mathscr{P}}_{h}:=\mathscr{P}_h(1)$ satisfies the $Q$-curvature type equation
\[
\mathscr{P}_h(\eta) = \eta^{b/a} Q_{\eta^{1/a} h}^{\mathscr{P}},\qquad\eta\in\mathcal{C}^\infty(M),\ \eta>0.
\]
If $\Delta_h$ denotes the Laplace-Beltrami operator and $R_h$ denotes the scalar curvature on $(M,h)$, the main examples of these operators are given by the conformal Laplacian $\mathscr{P}_h^1:= -\Delta_h + \frac{N-2}{4(N-1)}R_h$ with associated $Q$-curvature $Q^1_h=\mathscr{P}^1_h(1)=\frac{N-2}{4(N-1)}R_h$, and by its generalizations given by the GJMS \cite{GJMS1992} and the fractional conformal Laplacians \cite{CG2011,GZ2003}, which are a family of parametrized operators $\mathscr{P}_h^s:\mathcal{C}^\infty(M)\rightarrow\mathcal{C}^\infty(M)$ with $s\in(0,N/2)$. 
All of them satisfy the conformal law \eqref{VeryGeneralConformalLaw} with $a=\frac{N-2s}{4}$ and $b=\frac{N+2s}{4}$. The corresponding curvature $Q_h^s:=Q_h^{\mathscr{P}^s}$ is known as the (fractional) $Q$-curvature and the curvature equation is given by
\[
\mathscr{P}^s_h(\eta) = \eta^{\frac{N+2s}{N-2s}} Q_{\eta^{\frac{4}{N-2s}} h}^s .
\]
The case of the sphere with its round metric, which we denote by $(\mathbb{S}^N,g)$, is particularly interesting because the stereographic projection maps  the conformal operators $\mathscr{P}_g^s$ into the $s$-power of the Laplacian, $(-\Delta)^s$, in $\mathbb{R}^N$. 

In this paper, we explore a different kind of conformal operator on $\mathbb{S}^N$ that is linked, via the stereographic projection, to the logarithmic Laplacian in $\mathbb{R}^N$, denoted $L_{\Delta}$. The operator $L_\Delta$ was introduced in \cite{CW19} and can be seen as the derivative of $(-\Delta)^s$ with respect to the order parameter $s$, namely, $L_{\Delta} u=\partial_{s}(-\Delta)^s u|_{s=0}$.  Explicitly, it has the following pointwise evaluation in terms of a singular integral for a suitable function $v$,
\begin{align}\label{loglapdef}
L_\Delta v(x) = c_N\, \int_{B_1(x)} \frac{v(x) - v(y)}{|x - y|^N} \, dy
- c_N \int_{\mathbb{R}^N \setminus B_1(x)} \frac{v(y)}{|x - y|^N} \, dy
+ \rho_N\, v(x).    
\end{align}
Here  \( B_1(x) \) is the open ball in \( \mathbb{R}^N \) of radius 1 centered at \( x \), and $c_N:= \pi^{-\frac{N}{2}}\Gamma\left( \tfrac{N}{2} \right)$, $\rho_N:=2\ln 2 +\psi(\frac{N}{2})-\gamma$,  $\gamma:=-\Gamma'(1)$, are constants defined via the Gamma-function $\Gamma$.
This operator has attracted a lot of attention recently and it has been studied from several perspectives: regularity, numerical approximations, spectral properties, local extension problems (in the spirit of the Caffarelli-Silvestre extension), nonlinear equations, and also as a tool to characterize the $s-$dependence of fractional problems.  We refer to \cite{HS22,JSW20,CS24,AS23,HLS25,JWS24,HJSS25,CW19,LW21,CV24,FJW22,chw23,FJ23,CZ24,CV23,djf24,C25,AGV25,PS25} and the references therein for some of the latest results in this setting. 

In our next result, we show that one can define the conformal logarithmic Laplacian on the sphere in an analogous way. For $s\in(0,1)$ and $N\in \mathbb N$, the conformal fractional Laplacian on the sphere $(\mathbb{S}^N,g)$ is given by
\begin{align}\label{Definition:FractionalLaplacianSphere}
\mathscr{P}^s_{g}u(z)=c_{N,s}\text{P.V.}\int_{\mathbb{S}^N}\frac{u(z)-u(\zeta)}{\vert z-\zeta\vert^{N+2s}} dV_g(\zeta) + A_{N,s}u(z),\qquad u\in \mathcal{C}^\infty(\mathbb{S}^N),
\end{align}
where $\vert\cdot\vert$ denotes the Euclidean norm in $\mathbb{R}^{N+1}$
\begin{align}\label{Definition:COnstantAs}
A_{N, s}:=\frac{\Gamma\left(\frac{N}{2}+s\right)}{\Gamma\left(\frac{N}{2}-s\right)},\qquad 
c_{N,s}:= 4^s\pi^{-\frac{N}{2}}\frac{\Gamma(\frac{N}{2}+s)}{\Gamma(2-s)}s(1-s).
\end{align}
See, for instance \cite{CFS25,DM18}.

\begin{theorem}\label{main:thm}
Let $N\in \mathbb N$ and $u \in C^\beta(\mathbb{S}^N)$ for some $\beta > 0$. Then
\begin{equation}\label{Identity:LogarithmicConformalLaplacian}
\mathscr{P}^{\log}_g u(z) := \left.\frac{d}{ds}\right|_{s=0} [\mathscr{P}^s_{g}u](z)
=c_N \int_{\mathbb{S}^N}\frac{u(z) - u(\zeta)}{\vert z - \zeta \vert^N} dV_g(\zeta) + A_N u(z),
\end{equation}
where 
\begin{equation}\label{Identity:Constants}
c_N:= \pi^{-\frac{N}{2}}\Gamma\left( \tfrac{N}{2} \right),\qquad 
A_N:= 2 \psi\left(\tfrac{N}{2}\right),
\end{equation}
and $\psi(x):=\frac{\Gamma'(x)}{\Gamma(x)}$ denotes the Digamma function. Moreover,
\begin{itemize}
    \item[(i)] For $1\leq  p \leq \infty$, we have $\mathscr{P}^{\log}_g u \in L^p(\mathbb{S}^N)$ and $\frac{\mathscr{P}^{s}_g u - u}{s} \to \mathscr{P}^{\log}_g u$ in $L^p(\mathbb{S}^N)$ as $s \to 0^+$;
    \item [(ii)] $\Phi$ is an eigenfunction of the Laplace-Beltrami operator on the sphere associated to the eigenvalue $\lambda$,
    \begin{align*}
        (-\Delta)_g \Phi = \lambda \Phi \quad \text{ on }\S,
    \end{align*}
    if and only if $\Phi$ is an eigenfunction of $\mathscr{P}^{\log}_g$ associated to the eigenvalue $\varphi_N(\lambda):= 2 \psi\left(\frac{1}{2}+\sqrt{\frac{1}{4} (N-1)^2+\lambda}\right)$,
    \begin{align*}
        \mathscr{P}^{\log}_g \Phi = \varphi_N(\lambda) \Phi \quad \text{ on }\S.
    \end{align*}
    Furthermore, the function $\lambda\mapsto \varphi_N(\lambda)$ is strictly increasing in $(0,\infty)$.
\end{itemize}
\end{theorem}

Theorem~\ref{main:thm} completely characterizes all the eigenfunctions and eigenvalues of the conformal logarithmic Laplacian $\mathscr{P}^{\log}_g$ in the sphere. Recall that the eigenvalues of the Laplace-Beltrami operator $(-\Delta)_g$ are given by
\begin{align}\label{Equation:ExplicitEigenvalues}
    b_i:=i(i+N-1),\qquad i\in \mathbb N_0:=\mathbb N\cup \{0\}.
\end{align}
For $i\geq 2$, the corresponding eigenspaces are spanned by $c_i:= \binom{N+i}{N}-\binom{N+i-2}{N}$ orthonormal smooth real-valued spherical harmonics $Y_{i,j}\in L^2(\S)$ for $j=1,\ldots,c_i$, namely,
\begin{align}\label{Problem:CompleteEigenvalueSphere}
    -\Delta_{g}Y_{i,j} = b_i Y_{i,j}\quad \text{ in }\S \quad \text{ for } j=1,\ldots,c_i.
\end{align}
For the exceptional cases $i\in\{0,1\}$, we have $c_0:=1$, $c_1:=N+1$. Moreover, $Y_{0,1}$ is a constant function and $Y_{1,j}$ is simply the trace on the sphere of linear functions. Since $\lambda\mapsto \varphi_N(\lambda)$ is a strictly increasing function (see \eqref{si}), we have that $\varphi_N(b_i)<\varphi_N(b_j)$ for all $i<j$. Moreover, the first eigenvalue $\varphi_N(0)=2 \psi\left(\frac{N}{2}\right)=A_N$ is negative for $N=1,2$ and positive for $N\geq 3,$ whereas the second eigenvalue $\varphi_N(N)=2\psi(\frac{N+2}{2})>0$ is always positive\footnote{The digamma function satisfies the inequality $\psi(x)>\ln x - \frac{1}{x}$ for $x>0;$ see, for instance, \cite[inequality (2.2)]{A97}.}.

The following result extends \cite[Proposition 2.2]{CW19} to our setting and it characterizes the class of functions $u$ for which $\mathscr{P}^{\log}_g$ is well-defined and continuous.
\begin{proposition}\label{dini:prop}
Let $u\in L^1(\S)$ be Dini continuous at some $z\in \S$. Then $\mathscr{P}^{log}_{g}u(z)$ is well defined.  Moreover, if $u$ is uniformly Dini continuous, then $\mathscr{P}^{log}_{g}u$ is continuous on $\S$. 
\end{proposition}

Next we establish a useful relationship between the conformal logarithmic Laplacian $\mathscr{P}_{g}^{\log}$ and the logarithmic Laplacian in $\rn$ denoted by $L_\Delta$. For this, consider the stereographic projection (with respect to the south pole) $\sigma:\mathbb{S}^N\backslash\{-e_{N+1}\}\rightarrow\mathbb{R}^N$ given by 
\begin{align}\label{sigma}
z=(z',z_{N+1})\mapsto \sigma(z):=\frac{z'}{1+z_{N+1}},
\qquad\text{ and }\qquad x\mapsto \sigma^{-1}(x)= \left(\frac{2x}{1+\vert x\vert^2}, \frac{1-\vert x\vert^2}{1+\vert x\vert^2}\right). 
\end{align}
where $-e_{N+1}$ is the south pole on the sphere $\mathbb{S}^N$.  For a function $u:\mathbb{S}^N\smallsetminus\{-e_{N+1}\}\to \r$, we define a map $\iota:u\mapsto\iota(u)$ with $\iota(u):\mathbb{R}^N\rightarrow\mathbb{R}$ given by
\begin{align}\label{iotadef}
\iota(u)(x):=\phi(x)^{\frac{N}{2}}u(\sigma^{-1}(x))\quad \text{ for }x\in \rn,\qquad \phi(x):=\frac{2}{1+\vert x\vert^2}.
\end{align}

\begin{proposition}\label{Prop:ConformalWithEuclidean}
Let $u\in C^\infty(\mathbb{S}^N)$ and $v := \iota(u)$, then
\begin{align}\label{conf}
    \iota(\mathscr{P}_g^{\log} u)(x)=\phi(x)^{\frac{N}{2}}\mathscr{P}_g^{\log} u(\sigma^{-1}(x)) = L_\Delta v(x)-2v(x)\ln \phi(x)\qquad \text{ for all }x\in \rn.
\end{align}
\end{proposition}
Formula \eqref{conf} contrasts with the local case and the fractional case, where the connection between $\mathscr{P}_g^s$ and $(-\Delta)^s$ does not involve additive terms (see \eqref{lem:1} below). The term \(-2v(x)\ln \phi(x)\) in \eqref{conf} highlights the logarithmic nature of the operator from the perspective of conformal invariance. Heuristically, while the power of a product of real numbers distributes multiplicatively, as in \( |\lambda \mu|^\beta = |\lambda|^\beta |\mu|^\beta \), the logarithm introduces additivity: \( \ln |(\lambda \mu)^\beta| = \beta (\ln |\lambda| + \ln |\mu|) \).


As a corollary of Theorem~\ref{main:thm} and Proposition~\ref{Prop:ConformalWithEuclidean}, we show that spherical harmonics induce eigenfunctions in the whole space for the logarithmic Laplacian plus a certain potential. 
\begin{corollary}\label{iotaF:prop}
Let $u\in C^\infty(\mathbb{S}^N)$ be a spherical harmonic and let $v=\iota(u)$ with $\iota$ as in \eqref{iotadef}.  There is $\lambda\in \R$ such that
\begin{equation}\label{E1}
\mathscr{P}_g^{\log} u(z) = \lambda u\qquad \text{ on }\S
\end{equation}
and
\begin{equation}\label{E2}
L_\Delta v+v \ln\phi^{-2} = \lambda v\qquad \text{ in }\rn.
\end{equation}
\end{corollary}

The previous result can be applied to any closed-form expression for spherical harmonics, yielding explicit solutions to linear equations involving the logarithmic Laplacian in $\rn$ (see Remark~\ref{explicit:rmk}). Such formulas are of particular interest, as very few explicit solutions are currently known. See also Remark~\ref{constantsolsrmk} for a link between \eqref{E2} and a Yamabe-type problem.

\medskip

Next, we introduce a functional framework to study weak solutions of nonlinear problems.  This is one of the main methodological contributions of this paper, since it is nontrivial to find a suitable Hilbert space that provides a meaningful setting for the weak formulations of nonlinear problems involving the logarithmic Laplacian in unbounded domains. In particular, one requires a suitable condition on the decay of functions at infinity to have a well-posed weak problem and good embedding properties. We show next that this condition can be stated in terms of an $L^2$-weighted norm with a logarithmic weight that belongs to the $A_2$ Muckenhoupt class (see Lemma~\ref{A2}). Let
\begin{align}\label{Dlogdef}
D^{log}(\rn)
:=\left\{
v\in L^2(\rn)\::\: 
\|v\|<\infty
\right\}
\end{align}
endowed with the norm
\begin{align}\label{normdef}
\|v\|:=\left(\cE(v,v)+\int_{\rn}v(x)^2\ln(e+|x|^2)\, dx\right)^\frac{1}{2},
\end{align}
induced by the inner product
\begin{align}\label{sp}
\langle v_1,v_2 \rangle:=\cE(v_1,v_2)+\int_{\rn}v_1(x)v_2(x)\ln(e+|x|^2)\, dx,
\end{align}
where 
\begin{align*}
\cE(v_1,v_2):=\frac{c_N}{2}\int_{\rn}\int_{B_1(x)}\frac{(v_1(x)-v_1(y))(v_2(x)-v_2(y))}{|x-y|^N}\, dy\, dx.
\end{align*}

The following result establishes some of the main properties of this space. We recall that 
\begin{align}\label{cELdef}
\cE_L(v_1,v_2):=\cE(v_1,v_2)-c_N\iint_{\vert x - y\vert\geq 1}\frac{v_1(x)v_2(y)}{\vert x-y\vert^N}\; dxdy + \rho_N\int_{\mathbb{R}^N}v_1 v_2 \ dx
\end{align}
is the bilinear form associated to the logarithmic Laplacian $L_\Delta$, i.e., 
\begin{align*}
\cE_L(v_1,v_2)=\int_{\rn} v_2L_\Delta v_1 \, dx=\int_{\rn} v_1 L_\Delta  v_2\, dx\qquad \text{ for }v_1,v_2\in C^\infty_c(\rn),
\end{align*}
see \cite{CW19}. As shown in \cite{CW19}, the bilinear form $\cE_L(\cdot,\cdot)$ is not positive definite (hence it is not an inner product).

\begin{theorem}\label{Dlogprop}
The space $D^{log}(\rn)$ is a Hilbert space with the inner product \eqref{sp}. Moreover, 
\begin{enumerate}
    \item $C_c^\infty(\mathbb{R}^N)$ is dense in $D^{log}(\rn)$,
    \item $D^{log}(\rn)$ is compactly embedded in $L^2(\rn)$,
    \item there is $C>0$ such that, for every $v_1,v_2\in D^{log}(\rn)$, each term in \eqref{cELdef} is finite and
    \begin{align*}
        |\cE_L(v_1,v_2)|\leq C\|v_1\|\|v_2\|.
    \end{align*}
\end{enumerate}
\end{theorem}
To establish these properties, we rely on two main ingredients. The first one is \textit{Pitt's inequality},
\begin{align*}
\cE_L(v,v)+\int_{\mathbb{R}^N}\ln(|x|^{2})|v(x)|^2 \,dx 
\geq a_N\|v\|_{L^2(\rn)}^2,\qquad a_N:=2 \psi\left(\tfrac{N}{4}\right)+2\ln (2),\qquad v\in C^\infty_c(\rn)
\end{align*}
(see Proposition~\ref{Pitt:prop}); the second one is the norm
\begin{align*}
\|v\|_{\D} := \|u\|_{\H}, \qquad \text{ if }v=\iota(u) \text{ for some }u\in C^\infty(\S),
\end{align*}
where \( \iota \) is defined in~\eqref{iotadef} and \( \H \) denotes the logarithmic Sobolev space on the sphere \( \mathbb{S}^N \), see~\cite{FKT20}. This space is the completion of \( C^\infty(\mathbb{S}^N) \) with respect to the norm
\begin{equation} \label{NormLogSobolevSphere}
\|u\|_{\mathbb{H}(\mathbb{S}^N)} := \left( \frac{c_N}{2} \int_{\mathbb{S}^N} \int_{\mathbb{S}^N} \frac{(u(z) - u(\zeta))^2}{|z - \zeta|^N} \, dV_g(z) \, dV_g(\zeta) + \kappa \int_{\mathbb{S}^N} u^2 \, dV_g \right)^{1/2},
\end{equation}
where \( c_N \) as in~\eqref{Identity:Constants} and, for convenience,
\begin{align}\label{kappahyp}
\kappa>\max\{2(\psi(\tfrac{N}{2})
-\psi\left(\tfrac{N}{4}\right)),A_N,1\}
\end{align}
(see Theorem \ref{Ddefthm} and Remark \ref{posdef:rmk}).
 Furthermore, we have the characterization
\begin{align*}
\mathbb{H}(\mathbb{S}^N) := \overline{C^\infty(\mathbb{S}^N)}^{\|\cdot\|_{\mathbb{H}(\mathbb{S}^N)}} = \left\{ u \in L^2(\mathbb{S}^N) \,:\, \|u\|_{\mathbb{H}(\mathbb{S}^N)} < \infty \right\}.
\end{align*}

In Theorem~\ref{thmD}, we show that the norms \( \|\cdot\| \) and \( \|\cdot\|_{\D} \) are equivalent. This equivalence enables us to transfer some of the properties of \( \mathbb{H}(\mathbb{S}^N) \) to the space \( D^{\log}(\mathbb{R}^N) \). The constants in \eqref{NormLogSobolevSphere} are chosen to simplify some estimates (see Remark \ref{posdef:rmk}, for example).

In this sense, the space $\D$ (and, by equivalence, the space \( D^{\log}(\mathbb{R}^N) \)) can be seen as a logarithmic counterpart of the homogeneous fractional Sobolev space \( D^s(\mathbb{R}^N) \) for \( N > 2s \), or, in the local case, 
\[
D^1(\mathbb{R}^N) := \left\{ u \in L^{\frac{2N}{N-2}}(\mathbb{R}^N) \,:\, |\nabla u| \in L^2(\mathbb{R}^N) \right\},\qquad N \geq 3,
\]
which are isometrically isomorphic to the corresponding Hilbert spaces on the sphere $H_g^s(\mathbb{S}^N)$ and $H^1_g(\mathbb{S}^N)$; see, e.g.,~\cite{CFS25,CFS21,CSS21}. We believe that this framework provides a natural setting for the study of nonlinear problems with the logarithmic Laplacian in unbounded domains.

We also note that the norm \( \|\cdot\|_{\D} \) admits an explicit expression (see Theorem~\ref{Ddefthm} and Remark \ref{for:rmk}); however, it involves a combination of positive and negative terms, which poses some difficulties in certain limiting arguments.

\medskip

Now we use the space $D^{log}(\rn)$ to study the $Q$-curvature problem associated with the conformal logarithmic Laplacian and the logarithmic Yamabe problem in $\rn$.

Given a conformal metric $\widetilde{g}:=\eta g$, $\eta\in\mathcal{C}^\infty(\mathbb{S}^N)$, we define the \emph{logarithmic $Q$-curvature} as
\begin{equation}\label{Definition:LogCurvature}
Q^{\log}_{\tilde{g}}:=\mathscr{P}_{\tilde{g}}^{\log}(1).
\end{equation}
In the standard round metric $g$, by \eqref{Identity:LogarithmicConformalLaplacian} and \eqref{Identity:Constants}, we have the explicit expression of this curvature
\[
Q_g^{\log}= A_N= 2\psi(\tfrac{N}{2}).
\]

Next, we use the conformal law of the conformal logarithmic Laplacian.
\begin{proposition}\label{ConformalProperty}
Let $N\in \mathbb N$ and $\eta\in\mathcal{C}^{\infty}(\mathbb{S}^N)$ such that $\eta>0$. Then,
\begin{align}\label{Identity:LogConformalProperty}
    \mathscr{P}_{\eta g}^{\log}(\varphi) = \eta^{-\frac{N}{4}}\mathscr{P}_g^{\log}(\eta^{\frac{N}{4}}\varphi) - \varphi\ln \eta\qquad \text{for any $\varphi\in\mathcal{C}^\infty(\mathbb{S}^N)$}.
\end{align}
\end{proposition}

For conformal diffeomorphisms of the round sphere, the transformation law \eqref{Identity:LogConformalProperty} is given in weak form in \cite[Lemma~2]{FKT20}.
Here we show \eqref{Identity:LogConformalProperty} for an arbitrary smooth
positive conformal factor with a different proof. Now, taking $\tilde{g}:=u^{\frac{4}{N}}g$  in \eqref{Identity:LogConformalProperty}, with $u\in\mathcal{C}^\infty(\mathbb{S}^N)$ positive, we obtain that
\[
\mathscr{P}^{\log}_{\tilde{g}}(\varphi) = u^{-1}\mathscr{P}_g^{\log}(u\varphi) - \frac{4}{N}\varphi\ln u\quad \text{ on }\S.
\]
Hence, for $\varphi=1$, $Q^{\log}_{\tilde{g}}:=\mathscr{P}^{\log}_{\tilde{g}}(1) = u^{-1}\mathscr{P}_g^{\log}(u) - \frac{4}{N}\ln u$, and then 
\[
\mathscr{P}_g^{\log}(u) = \frac{4}{N}u\ln u + u Q^{\log}_{\tilde{g}}\quad \text{ on }\S.
\]
In analogy to the fractional case, the \emph{constant logarithmic $Q$-curvature} problem (or \emph{logarithmic Yamabe problem}) consists in finding a conformal metric $\tilde{g}=u^{\frac{4}{N}}g$ in such a way that $Q_{\tilde{g}}^{\log}\equiv \mu$ is constant.  As discussed earlier, this is equivalent to the existence of a positive solution to the \emph{logarithmic Yamabe equation} on the round sphere $(\mathbb{S}^N,g)$, namely, for $\mu\in \r$, 
\begin{equation}\label{Y1}
\mathscr{P}_g^{\log} u= \frac{4}{N}u\ln|u| + \mu u \quad \text{on} \ \; \mathbb{S}^N,\qquad u\in\mathbb{H}(\mathbb{S}^N).
\end{equation}

We say that $u\in\mathbb{H}(\mathbb{S}^N)$ is a  weak solution of \eqref{Y1} if
\begin{align*}
\frac{c_N}{2}\int_{\mathbb{S}^N}\int_{\mathbb{S}^N} \frac{(u(z)-u(\zeta))(\varphi(z)-\varphi(\zeta))}{\vert z - \zeta\vert^N} \;dV_g(z)\;dV_g(\zeta) + A_N\int_{\mathbb{S}^N}u \varphi\; dV_g 
= \int_{\S}\left(\frac{4}{N}u\ln|u| + \mu u\right)\varphi\; dV_g 
\end{align*}
for all $\varphi\in \mathbb{H}(\mathbb{S}^N)$. 

Now we can use Proposition~\ref{Prop:ConformalWithEuclidean} to link \eqref{Y1} with the corresponding logarithmic Yamabe problem in $\rn$, namely,
\begin{equation}\label{Y2}
L_\Delta v = \frac{4}{N} v \ln|v| + \mu v\quad \text{in}\ \mathbb{R}^N,\qquad v\in D^{log}(\mathbb{R}^N).
\end{equation}

We say that $v\in D^{log}(\mathbb{R}^N)$ is a  weak solution of \eqref{Y2} if
\begin{align*}
\cE_L(v,\vartheta)
= \int_{\rn}\left(\frac{4}{N}\ln|v| + \mu \right)v\vartheta\qquad \text{ for all }\vartheta\in D^{log}(\mathbb{R}^N).
\end{align*}

The following result establishes a one-to-one correspondence between weak solutions of \eqref{Y1} and of \eqref{Y2}.
\begin{theorem}\label{equivalentProblems}
Let $N\in \mathbb N$, $u:\mathbb{S}^N\to \r,$ and $v:\mathbb R^N\to \r$ be such that $v=\iota(u)$ with $\iota$ as in \eqref{iotadef}. Then $u$ is a  weak solution of \eqref{Y1}
if and only if $v$ is a weak solution of \eqref{Y2}.
\end{theorem}

In particular, the case of \emph{nonnegative solutions} has been carefully studied in the previous literature and all such solutions have been classified. To be more precise, in \cite[Theorem 1]{FKT20}, it is stated that, if $u$ is a nontrivial nonnegative weak solution of 
\begin{align}\label{frankprob}
\int_{\mathbb{S}^N}\frac{u(z)-u(\zeta)}{|z-\zeta|^N}\, dV_g(\zeta) = C_N u(z)\ln u(z)\quad \text{ for }z\in \mathbb{S}^N,\qquad C_N:=\frac{4 \pi^\frac{N}{2}}{N\Gamma(\frac{N}{2})},
\end{align}
then,
\begin{align}\label{franksol}
u(z)=u_{\theta}(z)=\left(\frac{\sqrt{1-|\theta|^2}}{1-\theta\cdot z}\right)^\frac{N}{2}\qquad \text{for some $\theta\in \mathbb R^{N+1}$ with $|\theta|<1.$}
\end{align}

We can now use Theorem~\ref{equivalentProblems} to extrapolate this classification result to weak solutions of \eqref{Y2}. 
\begin{theorem}\label{classthm}
If $v$ is a nonnegative nontrivial weak solution of \eqref{Y2}, then 
\begin{align}
v(x)=e^{\frac{N}{4}(A_N-\mu)}\left(\frac{2t}{t^2+|x-a|^2}\right)^\frac{N}{2}\qquad \text{for some $t>0$ and $a\in \rn$,}
\end{align}
where $A_N$ is given in \eqref{Identity:Constants}.
\end{theorem}
 
In the context of \emph{classical solutions}, this classification result is shown in \cite{CZ24}  using directly the method of moving planes (see Remark~\ref{chen:rmk}). In \cite{CZ24}, a classical solution is a Dini continuous function in $\rn$ such that $\int_{\rn}\frac{|u(x)|}{(1+|x|)^N}\, dx<\infty$ and satisfies the equation pointwisely.  The previous result shows that this classification also holds for weak solutions in $D^{log}(\mathbb{R}^N)$ and that it can be seen as a consequence of the conformal properties of $\mathscr{P}_g^{\log}$ and the results in \cite{FKT20}.

\medskip

We close this introduction with two remarks. First, we mention the very recent preprint \cite{chen2025logarithmic}, where logarithmic-type Laplacians on general Riemannian manifolds, denoted $\log(-\Delta),$ are defined via a Bochner integral and a heat semigroup. This operator is different from our conformal logarithmic Laplacian in the case of the sphere, since $\log(-\Delta)$ does not satisfy the conformal law \eqref{Identity:LogConformalProperty} and its spectrum is different (cf. \cite[Section 3.1]{chen2025logarithmic} with Theorem~\ref{main:thm}).

Second, our approach also yields a conformal logarithmic Laplacian on
\emph{cylinders}. We present the construction in Section~\ref{app:cyl} and we
thank the anonymous referee for suggesting this extension.

\medskip

The paper is organized as follows. In Section~\ref{prop:sec}, we introduce the conformal logarithmic Laplacian on the sphere and establish its main properties. This section includes the proofs of Theorem~\ref{main:thm}, Propositions~\ref{dini:prop},~\ref{ConformalProperty},~\ref{Prop:ConformalWithEuclidean}, and Corollary~\ref{iotaF:prop}. Section~\ref{fun:sec} is dedicated to the functional framework for the Yamabe-type problems; it contains the proofs of Theorems~\ref{Ddefthm} and~\ref{Dlogprop}, the latter following from the more general Theorem~\ref{Dlogprop2}. In Section~\ref{Yamabe:sec}, we present our equivalence result, Theorem~\ref{equivalentProblems} and a proof of Theorem~\ref{classthm} whereas Section \ref{app:cyl} is devoted to the derivation of the conformal logarithmic Laplacian on the cylinder. The paper concludes with an appendix providing details on some density results.

\section{The conformal logarithmic Laplacian on the sphere and its properties}\label{prop:sec}

First, we show that $\mathscr{P}^s_{g}u$ is well defined on $C^\beta(\S)$, the $\beta$-Hölder continuous functions on $\S$. 

Denote the geodesic distance in $\mathbb{S}^N$ by $d_g$ and recall that $|\cdot|$ is the Euclidean distance in $\mathbb R^{N+1}$. Note that there is $D_1>0$ such that
\begin{equation}\label{ComparisonGeodesicCordalDistance}
\frac{1}{D_1}\vert z - \zeta \vert \leq d_g(z,\zeta)\leq D_1 \vert z - \zeta\vert\qquad \text{ for all }\;z,\zeta\in\mathbb{S}^N.
\end{equation}
Let $r_0>0$ be the injectivity radius of $\mathbb{S}^N$. For 
\begin{align}\label{rhodef}
0<\rho<\min\{r_0,1\},    
\end{align}
we denote by $B_g(z,\rho)$ the geodesic ball of radius $\rho>0$ centered at $z\in\mathbb{S}^N$ and $B_{\rho}(0)$ denotes the ball of radius $\rho$ centered at the origin in $\mathbb{R}^N\equiv T_z\S$. Then, for any $z\in\mathbb{S}^N$, the exponential map $\exp_z: B_\rho(0)\rightarrow B_g(z,\rho)$, is a diffeomorphism and its inverse  $h_z:=\exp_z^{-1}:B_g(z,\rho)\rightarrow B_{\rho}(0)$ is a chart. If $(g_{jk})$ are the components of the metric $g$ in a coordinate chart, denote by $\vert g\vert:=\det (g_{jk})$. As $\mathbb{S}^N$ is compact, there exist positive constants $D_3, D_4>0$  such that
\begin{align}\label{e1}
\frac{1}{D_3}\leq\sqrt{\vert g\vert}\circ h_{z}^{-1}(y)\leq D_3\qquad \text{ for all }\;z\in\mathbb{S}^N,\ y\in B_{\rho}(0),
\end{align}
and
\begin{align}\label{e2}
\frac{1}{D_4}\vert  \exp^{-1}_z(\zeta)\vert \leq d_g(z,\zeta)\leq D_4 \vert \exp^{-1}_z(\zeta)\vert\qquad \text{ for all }\; z\in\mathbb{S}^N,\ \zeta\in B_g(z,\rho).    
\end{align}

Let $u : \S \to \mathbb{R}$ be a measurable function. The modulus of continuity $\omega_{u,z} : (0,\infty) \to [0,\infty)$ of $u$ at a point $z \in \S$ is defined by 
\begin{align*}
\omega_{u,z}(r) = \sup_{\substack{\zeta \in \S,\\ d_g(z,\zeta) \leq r}} |u(z) - u(\zeta)|.    
\end{align*}
The function $u$ is called Dini continuous at $z$ if $\int_0^1 \frac{\omega_{u,z}(r)}{r} \, dr <\infty$. Let $\omega_{u}(r) := \sup_{z \in \S} \omega_{u,z}(r).$ If $\int_0^1 \frac{\omega_{u}(r)}{r} \, dr < \infty$, then $u$ is called uniformly Dini continuous in $\S$.

We are ready to show  Proposition~\ref{dini:prop}.
\begin{proof}[Proof of Proposition~\ref{dini:prop}]
The proof follows the ideas in \cite[Proposition 2.2]{CW19} once a relationship between geodesic balls in the sphere and balls in $\rn$ is taken into consideration.  Here we sketch the proof for completeness. 

Let $u\in L^1(\S)$ be Dini continuous at some $z\in \S$. In the following, $C>0$ denotes a constant independent of $z\in \S.$  Using the notation in \eqref{e2}, let $\rho_1\in (0,\rho)$ be such that $D_4\rho_1<1$. By definition,
\begin{align*}
    \int_{B_g(z,\rho_1)}\frac{|u(z) - u(\zeta)|}{\vert z - \zeta \vert^N}\, dV_g(\zeta)
    \leq C\int_{B_g(z,\rho_1)}\omega_{u,z}(d_g(z,\zeta))d_g(z,\zeta)^{-N}\, dV_g(\zeta).
\end{align*}
Moreover, by \eqref{e1}, \eqref{e2}, and the fact that $t\mapsto \omega_{u,z}(t)$ is non-decreasing,
\begin{align*}
&\int_{B_g(z,\rho_1)} \omega_{u,z}(d_g(z,\zeta))d_g(z,\zeta)^{-N}dV_g(\zeta)
\leq C\int_{B_g(z,\rho_1)} 
\omega_{u,z}(D_4|\exp_z^{-1}(\zeta)|)|\exp_z^{-1}(\zeta)|^{-N}
\, dV_g(\zeta)\\
&=C\int_{B_{\rho_1}(0)}
\omega_{u,z}(D_4|\exp_z^{-1}\circ h_{z}^{-1}(y)|)|\exp_z^{-1}\circ h_{z}^{-1}(y)|^{-N}
\sqrt{\vert g\vert}\circ h_z^{-1}(y)\;dy\\
&\leq C\int_{B_{\rho_1}(0)} \omega_{u,z}(D_4 |y|)|y|^{-N}\;dy
= C \int_0^{\rho_1}\frac{\omega_{u,z}(D_4 t)}{|t|}\;dt
=  C \int_0^{D_4\rho_1}\frac{\omega_{u,z}(s)}{|s|}\;ds
\leq C \int_0^{1}\frac{\omega_{u,z}(s)}{|s|}\;ds<C.
\end{align*}
Hence, since $u\in L^1(\S)$, this implies that $\mathscr{P}^{log}_{g}u(z)$ is well defined.  

Furthermore, if $u$ is uniformly Dini continuous, then again using the relationship between geodesic balls in the sphere and balls in $\rn,$ one can argue the continuity of $\mathscr{P}^{log}_{g}u$ in $\S$ exactly as \cite[Proposition 2.2]{CW19} (it is, in fact, somewhat simpler since the sphere is a compact manifold).  We omit the details.
\end{proof}

\begin{proof}[Proof of Theorem~\ref{main:thm}]
Let $u \in C^\beta(\mathbb{S}^N)$ for some $\beta > 0$. Then, for $z\in \mathbb{S}^N$,
\begin{align*}
\lim_{s\to 0^+}\frac{\mathscr{P}^{s}_g u(z) - u(z)}{s}
&=\lim_{s\to 0^+}\frac{c_{N,s}}{s}\int_{\mathbb{S}^N}\frac{u(z)-u(\zeta)}{\vert z-\zeta\vert^{N+2s}} dV_g(\zeta)+\lim_{s\to 0^+}\frac{A_{N,s}- 1}{s}u(z)\\
&=c_{N}\int_{\mathbb{S}^N}\frac{u(z)-u(\zeta)}{\vert z-\zeta\vert^{N}} dV_g(\zeta)
+A_N u(z).
\end{align*}
Standard arguments can be used to show that, for $s$ sufficiently small, the function $f_s(z):=\int_{\mathbb{S}^N}\frac{u(z)-u(\zeta)}{\vert z-\zeta\vert^{N+2s}}\, dV_g(\zeta)$ is equicontinuous with respect to $s$ and uniformly bounded (arguing as in the proof of \eqref{eta:bdd} below, for instance). Then, Arzelà-Ascoli's theorem yields $(i)$ for $p=\infty$, and the case $1\leq p<\infty$ follows from the fact that $\S$ is a compact manifold. To argue $(ii)$, let $\Phi\in C^\infty(\mathbb S^N)$ be an eigenfunction of the Laplace-Beltrami operator on the sphere with eigenvalue $\lambda$. Then, for all $z\in \mathbb S^N$,
\begin{align}
\mathscr{P}^s_{g}\Phi(z)= \varphi_{N,s}(\lambda)\Phi(z),\qquad 
\varphi_{N,s}(\lambda) := \frac{\Gamma\left(\frac{1}{2}+s+\sqrt{\lambda+\left(\frac{N-1}{2}\right)^2}\right)}{\Gamma\left(\frac{1}{2}-s+\sqrt{\lambda+\left(\frac{N-1}{2}\right)^2}\right)},
\end{align}
see, for instance, \cite{DM18} or \cite[Lemma 2.6]{CFS25}.  Hence,
\begin{align*}
\mathscr{P}^{log}_{g}\Phi(z)
=\tfrac{d}{ds}|_{s=0}\mathscr{P}^s_{g}\Phi(z)
=\tfrac{d}{ds}|_{s=0}\varphi_{N,s}(\lambda)\Phi(z)
=\varphi_N(\lambda)\Phi(z). 
\end{align*}
 This shows that if $\Phi$ is an eigenfunction of the Laplace-Beltrami operator on the sphere associated to the eigenvalue $\lambda$, then $\Phi$ is an eigenfunction of $\mathscr{P}^{\log}_g$ associated to the eigenvalue $\varphi_N(\lambda)$. The converse can be argued as in \cite[Remark 2.9]{CFS25} using that $\varphi_N$ is strictly increasing, because
 \begin{align}\label{si}
     \varphi'_N(\lambda)=\frac{2 \psi ^{(1)}\left(\frac{1}{2}+\sqrt{\frac{1}{4} (N-1)^2+\lambda
   }\right)}{\sqrt{4 \lambda +(N-1)^2}}>0,
 \end{align}
 where $\psi^{(m)}$ is the polygamma function of order $m\geq 0$ and $\psi^{(1)}(\tau)=\int_0^\infty \frac{t e^{-\tau t}}{1-e^{-t}}\, dt>0$ in $(0,\infty)$.
\end{proof}

Next, we study the conformal invariance of the operator $\mathscr{P}^{\log}_g$. Let $\eta\in\mathcal{C}^{\infty}(\mathbb{S}^N)$ be positive and, for each $s\in(0,1)$, define $\eta_s:=\eta^{\frac{4}{N-2s}}.$ The fractional conformal Laplacian satisfies the following conformal law
\begin{equation}\label{Equation:ConformalProperty}
\mathscr{P}^s_{\eta_s g}(u) = \eta^{-\frac{N+2s}{N-2s}} \mathscr{P}_g^s(\eta u).
\end{equation}
See, for instance \cite{CFS25,CG2011,DM18}. Note that $\eta = \eta_{s}^{\frac{N-2s}{4}},$ and then \eqref{Equation:ConformalProperty} implies that $\mathscr{P}_{\eta g}^s(u) =  \eta^{-\frac{N+2s}{4}}\mathscr{P}_g^s(\eta^{\frac{N-2s}{4}}u).$

Next we show Proposition~\ref{ConformalProperty}.
\begin{proof}[Proof of Proposition~\ref{ConformalProperty}]
Let $\eta$ and $\varphi$ as in the statement. Note that
\begin{align*}
&\eta^{-\frac{N+2s}{4}}\mathscr{P}_g^s(\eta^{\frac{N-2s}{4}}\varphi)-\varphi\\
&=\left(\eta^{-\frac{N+2s}{4}}-\eta^{-\frac{N}{4}}\right)\mathscr{P}_g^s(\eta^{\frac{N-2s}{4}}\varphi)
+\eta^{-\frac{N}{4}}(\mathscr{P}_g^s-I)(\eta^{\frac{N-2s}{4}}\varphi)
+\eta^{-\frac{N}{4}}(
(\eta^{\frac{N-2s}{4}}-\eta^{\frac{N}{4}})
\varphi)
+\eta^{-\frac{N}{4}}(\eta^{\frac{N}{4}}\varphi)
-\varphi,
\end{align*}
then, by Theorem~\ref{main:thm} and \eqref{Equation:ConformalProperty},
\begin{align*}
 \mathscr{P}_{\eta g}^{\log}(\varphi)&=\lim_{s\to 0^+}\frac{\mathscr{P}^s_{\eta g}(\varphi)-\varphi}{s}
=
\lim_{s\to 0^+}\frac{\eta^{-\frac{N+2s}{4}}\mathscr{P}_g^s(\eta^{\frac{N-2s}{4}}\varphi)-\varphi}{s}\\
&=\lim_{s\to 0^+}\frac{\left(\eta^{-\frac{N+2s}{4}}-\eta^{-\frac{N}{4}}\right)}{s}\eta^{\frac{N}{4}}\varphi
+\eta^{-\frac{N}{4}}\lim_{s\to 0^+}\frac{(\mathscr{P}_g^s-I)}{s}(\eta^{\frac{N}{4}}\varphi)
+\eta^{-\frac{N}{4}} \varphi\lim_{s\to 0^+}\frac{\eta^{\frac{N-2s}{4}}-\eta^{\frac{N}{4}}}{s}\\
&=
-\frac{1}{2} \eta ^{-N/4} \ln (\eta )
\eta^{\frac{N}{4}}\varphi
+\eta^{-\frac{N}{4}}\mathscr{P}_g^{log}(\eta^{\frac{N}{4}}\varphi)
-\eta^{-\frac{N}{4}} \varphi
\frac{1}{2} \eta ^{N/4} \ln(\eta )=\eta^{-\frac{N}{4}}\mathscr{P}_g^{log}(\eta^{\frac{N}{4}}\varphi)
- \varphi\ln(\eta ).
\end{align*}
\end{proof}

The volume form in the coordinates given by the stereographic projection \eqref{sigma} is
\begin{align}\label{volumeform}
dV_g (z)
= \left(\frac{2}{1 + |x|^2} \right)^N dx 
= (\phi(x))^N dx,\qquad x=\sigma(z).
\end{align}
In particular,  a change of variables of a function $f\in L^1_g(\S)$ yields
\begin{align*}
\int_{\S}f(z)\, d V_g(z) = \int_{\rn} f(\sigma^{-1}(x))\phi(x)^N\, dx.
\end{align*}

Define $\phi_s(x):= \phi^{\frac{N-2s}{2}}(x) = \left(\frac{2}{1+|x|^{2}}\right)^{\frac{N-2s}{2}}.$ We have the following relationship. See, for instance, \cite{DM18} or \cite[Lemma~1.1]{CFS25}.
\begin{lemma}\label{lem:1}
    For $u\in C^\infty(\mathbb S^N)$, $v := u\circ \sigma^{-1}$, $z\in \mathbb S^N$, and $x=\sigma(z) \in \mathbb R^N$, it holds that
    \begin{align}\label{eq:2}
    \mathscr{P}^s_{g}u(z) = \phi_s(x)^{-\frac{N+2s}{N-2s}}(-\Delta)^s(\phi_sv)(x)
    \end{align}
\end{lemma}
Note that \eqref{eq:2} is basically a particular case of the intertwining rule \eqref{Equation:ConformalProperty}. The following result shows some properties of the map $\iota$.
\begin{lemma}\label{LogLap:lem}
Let $u\in C^\infty(\mathbb{S}^N)$ and let $v:\mathbb R^N\to \r$ be given by $v=\iota(u)$. 
Then $L_\Delta v \in L^2(\rn)$.  Moreover,  for $s\in (0,1/2)$,
$(-\Delta)^s(\phi^{\frac{N-2s}{2}} u\circ\sigma^{-1})\in L^2(\rn).$
\end{lemma}
\begin{proof}
Using that $u\in C^\infty(\S)$, we have
\begin{align*}
|v(x)|\leq \frac{C_1}{(1+|x|^2)^{\frac{N}{2}}}\qquad \text{ for }x\in \R^N    
\end{align*}
and
\begin{align*}
|Dv(x)|\leq \|u\|_{L_g^\infty(\S)}|D\phi(x)^{\frac{N}{2}}|+\|D_gu\|_{L_g^\infty(\S)}|\phi(x)^{\frac{N}{2}}D(\sigma^{-1})(x)|
\leq C_1\left(\frac{1}{1+|x|^{N+1}}+\frac{1}{1+|x|^{N}}\frac{1}{1+|x|^{2}}\right)
\leq \frac{C_2}{1+|x|^{N+1}}
\end{align*}
for $x\in \R^N$ and for some constants $C_1$ and $C_2$ depending only on  $\|u\|_{L_g^\infty(\S)},\|D_g u\|_{L_g^\infty(\S)},$ and $N$. Hence, $v\in L^p(\R^N)$ for all $p>1$ and $|Dv|\in L^2(\rn).$ Let $p\in (1,2)$ and $p'=\frac{p}{p-1}>2$, then, by the Hausdorff-Young inequality,
\begin{align}\label{widehatvp}
\|\widehat v\|_{L^{p'}(\R^N)}=\left(\int_{\R^N}|\widehat v|^{p'} \, dx\right)^\frac{1}{p'}\leq \left(\int_{\R^N}|v|^{p} \, dx\right)^\frac{1}{p}=\|v\|_{L^{p
}(\R^N)}.   
\end{align}

Moreover, by Parseval–Plancherel identity,
\begin{align}\label{vbd}
\int_{\rn}|\xi|^2|\widehat v(\xi)|^2\, d\xi
=\int_{\rn}|\widehat {Dv}|^2\, d\xi
=|Dv|_2^2<\infty.
\end{align}
Now, it is easy to verify that $v\in C^\infty(\rn)$ and
\begin{align*}
v\in L^1_0(\rn):=\left\{ w\in L^1_{loc}(\rn)\::\: \int_{\rn}\frac{|w(x)|}{(1+|x|)^N}\, dx<\infty  \right\}.
\end{align*}
Then, by \cite[Proposition 1.3]{CW19}, $L_\Delta v$ is well-defined in $\rn$.  Set $\ell(\xi):=\ln|\xi|^2$. By \eqref{widehatvp},
\eqref{vbd}, and H\"older's inequality,
\begin{align*}
\int_{\rn}|\ell(\xi)|^2|\widehat v(\xi)|^2\,d\xi
&=
\int_{|\xi|<1}|\ell(\xi)|^2|\widehat v(\xi)|^2\,d\xi
+
\int_{|\xi|\geq1}|\ell(\xi)|^2|\widehat v(\xi)|^2\,d\xi\\
&\leq
\|\widehat v\|_{L^{p'}(\rn)}^2
\left(
\int_{|\xi|<1}
|\ell(\xi)|^{\frac{2p'}{p'-2}}\,d\xi
\right)^{\frac{p'-2}{p'}}
+
C\int_{|\xi|\geq1}
|\xi|^2|\widehat v(\xi)|^2\,d\xi
<\infty.
\end{align*}
Here we used that $|\ell(\xi)|^2\leq C|\xi|^2$ for
$|\xi|\geq1$.

Let $\mathcal S'(\rn)$ denote the dual of the Schwarz space, namely, the space of tempered distributions and let $\langle \cdot, \cdot\rangle$ denote its dual pairing.  For every $\eta\in C_c^\infty(\rn)$, \cite[Theorem~1.1(ii)]{CW19} yields that
\begin{align*}
\langle L_\Delta v,\eta\rangle
&=\langle v,L_\Delta\eta\rangle
=\langle \widehat v,\ell\,\widehat\eta\rangle_{L^2(\rn)}=\left\langle
\mathcal F^{-1}(\ell\,\widehat v),\eta
\right\rangle_{L^2(\rn)}.
\end{align*}
Thus $L_\Delta v=\mathcal F^{-1}(\ell\,\widehat v)$ in $\mathcal S'(\rn).$ Since $\ell\,\widehat v\in L^2(\rn)$, it follows that
$L_\Delta v\in L^2(\rn)$.

Finally, for $s\in(0,1/2)$, Lemma~\ref{lem:1} gives
$(-\Delta)^s
\left(
\phi^{\frac{N-2s}{2}}u\circ\sigma^{-1}
\right)
=
\phi^{\frac{N+2s}{2}}
(\mathscr P_g^s u)\circ\sigma^{-1}.
$ Since $\mathscr P_g^s u\in C^\infty(\S)$, we have
$\left\|
(-\Delta)^s
\left(
\phi^{\frac{N-2s}{2}}u\circ\sigma^{-1}
\right)
\right\|_{L^2(\rn)}^2
\leq
\|\mathscr P_g^s u\|_{L_g^\infty(\S)}^2
\int_{\rn}\phi(x)^{N+2s}\,dx
<\infty.
$
\end{proof}

Using \eqref{eq:2} we can now show Proposition~\ref{Prop:ConformalWithEuclidean}. 
\begin{proof}[Proof of  Proposition~\ref{Prop:ConformalWithEuclidean}]
Let $u\in C^\infty(\mathbb S^N)$, $w := u\circ \sigma^{-1}$, $z\in \mathbb S^N$, and $x=\sigma(z) \in \mathbb R^N$. Then, by Lemma~\ref{lem:1},
    \begin{align}
    \mathscr{P}^s_{g}u(z) 
    = \phi_s(x)^{-\frac{N+2s}{N-2s}}(-\Delta)^s(\phi_s w)(x)
    = \phi(x)^{-\frac{N+2s}{2}}(-\Delta)^s(\phi^{\frac{N-2s}{2}} w)(x).\label{v0}
    \end{align}
Observe that
\begin{align}
&\phi^{-\frac{N+2s}{2}}(-\Delta)^s(\phi^{\frac{N-2s}{2}} w)-w \notag\\
&=\left(\phi^{-\frac{N+2s}{2}}-\phi^{-\frac{N}{2}}\right)(-\Delta)^s(\phi^{\frac{N-2s}{2}}w)
+\phi^{-\frac{N}{2}}((-\Delta)^s-I)(\phi^{\frac{N-2s}{2}}w)
+\phi^{-\frac{N}{2}}(
(\phi^{\frac{N-2s}{2}}-\phi^{\frac{N}{2}})
w)
+\phi^{-\frac{N}{2}}(\phi^{\frac{N}{2}}w)
-w.\label{v1}
\end{align}
Arguing as in \cite[Theorem 1.1]{CW19}, using Lemma~\ref{LogLap:lem} and the continuity of the Fourier transform $\mathcal F:L^2(\rn)\to L^2(\rn)$, it is easy to show that
\begin{align*}
\lim_{s \to 0^+} \frac{\left((-\Delta)^s - I \right)}{s}(\phi^{\frac{N-2s}{2}}w)
= \lim_{s \to 0^+}\mathcal F^{-1}\left[ \left( \frac{|\cdot|^{2s} - 1}{s} \right) \mathcal F(\phi^{\frac{N-2s}{2}}w)\right]
= \mathcal F^{-1}(2 \log |\cdot|\, \mathcal F(\phi^{\frac{N}{2}}w))=L_\Delta (\phi^{\frac{N}{2}}w)
\end{align*}
in $L^2(\mathbb{R}^N)$. Therefore, by Theorem ~\ref{main:thm}, \eqref{v0}, and \eqref{v1},
\begin{align*}
\mathscr{P}^{log}_{g}u(z)&=\lim_{s\to 0^+}\frac{\mathscr{P}^s_{g}u(z) - u(z)}{s}
=\lim_{s\to 0^+}\frac{\phi^{-\frac{N+2s}{2}}(-\Delta)^s(\phi^{\frac{N-2s}{2}}w)-w}{s}\\
&=\lim_{s\to 0^+}\frac{\phi^{-\frac{N+2s}{2}}(-\Delta)^s(\phi^{\frac{N-2s}{2}}w)-
\phi^{-\frac{N+2s}{2}}(\phi^{\frac{N-2s}{2}}w)}{s}
+\frac{\phi^{-\frac{N+2s}{2}}(\phi^{\frac{N-2s}{2}}w)}{s}
-\frac{w}{s}\\
&=\lim_{s\to 0^+}\frac{\phi^{-\frac{N+2s}{2}}(-\Delta)^s(\phi^{\frac{N-2s}{2}}w)-
\phi^{-\frac{N+2s}{2}}(\phi^{\frac{N-2s}{2}}w)}{s}
+\frac{(\phi^{-2s}-1)}{s}w\\
&=\phi^{-\frac{N}{2}}\lim_{s\to 0^+}\frac{((-\Delta)^s-I)}{s}(\phi^{\frac{N-2s}{2}}w)
+w\lim_{s\to 0^+}\frac{(\phi^{-2s}-1)}{s}
=\phi^{-\frac{N}{2}}L_\Delta(\phi^{\frac{N}{2}}w)
- 2w\ln(\phi ).
\end{align*}
The claim now follows by the definition of $\iota$.
\end{proof}

\begin{remark}
\emph{
Let $g$ be the round metric on the sphere and let $h$ be the canonical flat metric in the Euclidean space. Under the stereographic coordinates $g=\phi^2h$, if we substitute $\eta=\phi^2$  in \eqref{Identity:LogConformalProperty}, we obtain \eqref{conf}, namely
\[
\mathscr{P}_g^{\log} u =\mathscr{P}_{\phi^2 h}^{\log} u =\left(\phi^2\right)^{-\frac{N}{4}}\mathscr{P}_h^{\log}\left( (\phi^2)^{\frac{N}{4}}u\circ\sigma^{-1} \right) - u\circ\sigma^{-1}\ln\phi^2 = \phi^{-\frac{N}{2}}(x)L_\Delta(\phi^{\frac{N}{2}}u\circ\sigma^{-1})-2u\circ\sigma^{-1}\ln \phi(x),
\]
where $\mathscr{P}_h^{\log}=L_\Delta$.
}
\end{remark}

We are ready to show Corollary~\ref{iotaF:prop}.

\begin{proof}[Proof of Corollary~\ref{iotaF:prop}]
By Theorem~\ref{main:thm} there is $\lambda$ such that \eqref{E1} holds. The claim now follows by Proposition~\ref{Prop:ConformalWithEuclidean}, because $\lambda v=
\iota(\lambda u)= 
\iota (\mathscr{P}_g^{\log} u)
=L_\Delta v - 2 v \ln\phi$ in $\rn.$
\end{proof}

\begin{remark}\label{explicit:rmk}
\emph{
 Here we give a couple of simple examples of the use of Corollary~\ref{iotaF:prop}.
\begin{enumerate}
    \item \textbf{The constant spherical harmonic:} Let $u\equiv 1$, then $u$ solves \eqref{E1} with $\lambda_1=A_N$ and $v\equiv \phi^{\frac{N}{2}}$. Therefore,
\begin{align*}
L_\Delta (\phi^{\frac{N}{2}}) = \phi^{\frac{N}{2}}(A_N+\ln\phi^{2})\qquad \text{ in }\rn. 
\end{align*}
See also Remark~\ref{constantsolsrmk} for a link between this equation and a Yamabe-type problem.
\item \textbf{The first non-constant spherical harmonics:} Let $u_i(z)=z_i$ for $i=1,\ldots, N+1.$ Then $u_i$ solves \eqref{E1} with $\lambda=2\psi(\frac{N+2}{2})$ and (recall \eqref{sigma})
\begin{align*}
     v_i(x):=\iota(u_i)(x)=\phi(x)^\frac{N}{2}u_i(\sigma^{-1}(x))
     =2^{\frac{N+2}{2}}\frac{x_i}{(1+|x|^2)^{\frac{N+2}{2}}}\qquad \text{ for }i=1,\ldots,N,
\end{align*}
and 
\begin{align*}
     v_{N+1}(x):=\iota(u_{N+1})(x)=\phi(x)^\frac{N}{2}u_{N+1}(\sigma^{-1}(x))=2^{\frac{N}{2}}\frac{1-\vert x\vert^2}{(1+|x|^2)^{\frac{N+2}{2}}}.
\end{align*}
    Therefore, for $i=1,\ldots,N+1,$
    \begin{align*}
    L_\Delta v_i = \left(2\psi(\tfrac{N+2}{2}) + \ln\phi^{2}\right)v_i\qquad \text{ in }\rn.
    \end{align*}
\end{enumerate}
}
\end{remark}

\section{Functional settings for weak solutions}\label{fun:sec}

Recall the definition of the space $\mathbb{H}(\mathbb{S}^N)$ and of the norm $\|\cdot\|_{\mathbb{H}(\mathbb{S}^N)}$ given in the introduction.  Integrating by parts, 
\begin{align*}
\|u\|_{\mathbb{H}(\mathbb{S}^N)}=\left( \int_{\mathbb{S}^N}u\mathscr{P}_g^{\log}u \; dV_g + (\kappa-A_N)\int_{\mathbb{S}^N}u^2\; dV_g\right)^{1/2}
\qquad \text{for $u\in\mathcal{C}^\infty(\mathbb{S}^N)$}
\end{align*}
with $A_N$ as in \eqref{Identity:Constants}. Moreover, $\mathbb{H}(\mathbb{S}^N)$ is a Hilbert space with the inner product
\begin{equation}\label{InteriorProductLogSobolevSphere}
\langle u_1, u_2\rangle_{\mathbb{H}(\mathbb{S}^N)} := \frac{c_N}{2}\int_{\mathbb{S}^N}\int_{\mathbb{S}^N} \frac{(u_1(z)-u_1(\zeta))(u_2(z)-u_2(\zeta))}{\vert z - \zeta\vert^N} \;dV_g(z)\;dV_g(\zeta) + \kappa \int_{\mathbb{S}^N}u_1u_2\; dV_g.
\end{equation}
Observe that 
\begin{align}\label{AN:form}
\langle u_1, u_2\rangle_{\mathbb{H}(\mathbb{S}^N)}= \int_{\mathbb{S}^N}u_1\mathscr{P}_g^{\log}u_2 \; dV_g
+ (\kappa-A_N) \int_{\mathbb{S}^N}u_1u_2\; dV_g
\qquad \text{ for $u_1,u_2\in\mathcal{C}^\infty(\mathbb{S}^N)$.}
\end{align}

For our purposes, it is convenient to see that $C^\infty_c(\S\backslash \{-e_{N+1}\})$ is also dense in $\mathbb{H}(\mathbb{S}^N)$. Here $C^\infty_c(\S\backslash \{-e_{N+1}\})$ denotes the space of smooth functions on $\S$ with compact support in $\S\backslash \{-e_{N+1}\}$.

\begin{lemma}\label{density:lem2}
For $N\in \mathbb N$,
\begin{align}\label{Hchar}
\mathbb{H}(\mathbb{S}^N):=
 \overline{C^\infty(\mathbb{S}^N)}^{\|\cdot\|_{\mathbb{H}(\mathbb{S}^N)}}
  =\overline{C_c^\infty(\mathbb{S}^N\backslash\{-e_{N+1}\})}^{\|\cdot\|_{\mathbb{H}(\mathbb{S}^N)}}
=\left\{ u\in L^2(\mathbb{S}^N)\::\: \|u\|_{\mathbb{H}(\mathbb{S}^N)}<\infty\right\}.
\end{align}  
\end{lemma}
\begin{proof}
The space $\mathbb{H}(\mathbb{S}^N)$ is used in \cite{FKT20} (with a different notation and with an equivalent norm), where it is shown that 
$\overline{C^\infty(\mathbb{S}^N)}^{\|\cdot\|_{\mathbb{H}(\mathbb{S}^N)}}=\left\{ u\in L^2(\mathbb{S}^N)\::\: \|u\|_{\mathbb{H}(\mathbb{S}^N)}<\infty\right\}$. The density of $C^\infty(\mathbb{S}^N\backslash\{-e_{N+1}\})$, follows from Lemma~\ref{density:lem}, shown in the appendix.
\end{proof}

This result is very helpful, because
\begin{align}\label{Cc}
C^\infty_c(\rn)=\iota (C_c^\infty(\S\backslash\{-e_{N+1}\})).    
\end{align}

\subsection{A compact embedding}\label{embed:sec}

Let $\Omega\subset\mathbb{R}^N$ be an open, bounded, and Lipschitz domain. Let
\[
\mathbb{H}(\Omega):=\left\{v\in L^2(\mathbb{R}^N) \;:\;  \iint_{\vert x-y\vert\leq 1}\frac{\vert v(x) - v(y) \vert^2}{\vert x - y\vert^N} \;dxdy<\infty, \ v=0 \text{ in }\mathbb{R}^N\backslash\Omega \right\}.
\]
This space is introduced in \cite{CW19} to study logarithmic problems in bounded domains with Dirichlet conditions. It is a Hilbert space endowed with the inner product and norm 
\[
\mathcal{E}(v_1,v_2):=\frac{c_N}{2}  \iint_{\vert x-y\vert\leq 1}\frac{( v_1(x) - v_1(y) )( v_2(x) - v_2(y) )}{\vert x - y\vert^N} \;dxdy,\qquad 
\Vert v\Vert_{\mathbb{H}(\Omega)}:=\left(\mathcal{E}(v,v)\right)^{1/2}.
\]
Also recall that the logarithmic Laplacian $L_\Delta$ in $\mathbb{R}^N$ has the associated quadratic form $\cE_L(\cdot,\cdot)$ (see \eqref{cELdef}) which is well defined in $\mathbb{H}(\Omega)\times \mathbb{H}(\Omega)$.

Let $B_r(y)$ denote the ball of radius $r>0$ centered at $y\in\mathbb{R}^N$. By \cite[eq. (3.4)]{CW19}, 
\begin{equation}\label{NormLogOmega}
\Vert v\Vert_{\mathbb{H}(\Omega)}^2 = \frac{1}{2}\int_\Omega\int_\Omega (v(x) - v(y))^2K(x-y)\; dx dy + \int_\Omega v^2(x)\kappa_{\Omega}(x) \;dx,
\end{equation}
where $K:\mathbb{R}^N\backslash\{0\}\rightarrow \mathbb{R}$ and $\kappa_{\Omega}:\Omega\rightarrow\mathbb{R}$ are given by 
\[
K(x):=c_N\frac{1_{B_1(0)}(x)}{\vert x\vert^N},\qquad \kappa_\Omega(x):= c_N\int_{B_1(x)\backslash\Omega}\frac{1}{\vert x - y \vert^N} \; dy.
\]

The following result is shown in \cite[Theorem 2.1]{CDP18}, see also \cite[eq. (3.3)]{CW19} and \cite[Corollary 2.3]{LW21}.

\begin{lemma}\label{Lemma:CompactnessOmega}
The embedding $\mathbb{H}(\Omega)\hookrightarrow L^2(\Omega)$ is compact.
\end{lemma}

In this section, we use Lemma~\ref{Lemma:CompactnessOmega} to show the following.
\begin{proposition}\label{embed:prop}
The embedding $\mathbb{H}(\mathbb{S}^N)\hookrightarrow L_g^2(\mathbb{S}^N)$ is compact.
\end{proposition}

The idea of the proof is very natural: we apply Lemma~\ref{Lemma:CompactnessOmega} on a suitable finite covering of $\S$ and then glue everything adequately. To be more precise, since $\mathbb{S}^N$ is compact, there exists a finite number of charts $\{(U_i,h_i)\}_{i=1,\ldots,\ell}$ covering $\mathbb{S}^N$ which can be taken so that 
\begin{enumerate}
\item $h_i(U_i)=B_{1/2}(0)$ for each $i=1,\ldots,\ell.$
\item $h_i:U_i\rightarrow B_{1/2}(0)$ is smooth and bi-Lipschitz.
\end{enumerate}
Let $\{\alpha_i\}_{i=1,\ldots,\ell}$ be a partition of unity subordinated to the covering $\{U_i\}_{i=1}^{\ell}$. Since $\alpha_i$ has compact support in $U_i$, it follows that $\text{supp}(\alpha_i\circ h_i^{-1})\subset B_{1/2}(0)$ is compact and, in particular, there exists $r_i>0$ such that
\begin{equation}\label{InequalitySupport}
\vert x\vert\leq r_i<\frac{1}{2}\qquad \text{ for } x\in\text{supp}(\alpha_i\circ h_i^{-1}),\quad i=1,\ldots,\ell.
\end{equation}
If $(g_{jk})$ are the components of the metric $g$ in a coordinate chart, denote by $\vert g\vert:=\det (g_{jk})$. By compactness of the sphere, there is $C_0>0$ such that
\begin{equation}\label{InequalityMetric}
\frac{1}{C_0}\leq \sqrt{\vert g\vert}\circ h_i^{-1}(x)\leq C_0\quad\text{ for every } x\in B_{1/2}(0) \text{ and }i=1,\ldots,\ell.
\end{equation}
As $h_i^{-1}$ is Lipschitz continuous, for any $i=1,\ldots,\ell$, there is $C_1>0$ such that
\begin{equation}\label{InequalityLipschitz}
\vert h_i^{-1}(x) - h_i^{-1}(y) \vert \leq C_1 \vert x - y \vert\qquad \text{ for all } x\in B_{1/2}(0) \text{ and }i=1,\ldots,\ell.
\end{equation}

For $u\in\mathcal{C}^\infty(\mathbb{S}^N)$, let
\[
v_i:\mathbb{R}^N\rightarrow\mathbb{R},  \quad v_i(x):= \begin{cases}
 (\alpha_i u)\circ h_i^{-1}(x), & x \in B_{1/2}(0),\\
 0, & x\notin B_{1/2}(0).
 \end{cases}
\]

Since $\alpha_i$ and $h_i$ are smooth and $g$ is positive definite and smooth, each $v_i$ is smooth (note that $\text{supp}(\alpha_i\circ h_i^{-1})\subset B_{1/2}(0)$, $\text{supp}(v_i)$ is compact, and $v_i\equiv 0$ in $\mathbb{R}^N\backslash \overline{B}_{1/2}(0)$). Hence, for each $u\in\mathcal{C}^\infty(\mathbb{S}^N)$ and each $i=1,\ldots,\ell$, $v_i\in \mathbb{H}(B_{1/2}(0))$.

\begin{lemma} \label{Lemma:EstimatesLogNorm}
There is $C>0$ such that $\Vert v_i\Vert_{\mathbb{H}(B_{1/2}(0))}\leq C \Vert \alpha_iu\Vert_{\mathbb{H}(\mathbb{S}^N)}$ for any $u\in \mathcal{C}^\infty(\mathbb{S}^N)$ and $i=1,\ldots,\ell$.
\end{lemma}

\begin{proof}
Fix $i\in\{1,\ldots,\ell\}$. First we make some estimates on $\kappa_{B_{1/2}(0)}.$ As $\text{supp}(v_i)\subset\text{supp}(\alpha_i\circ h_i^{-1})\subset B_{1/2}(0)$, by inequality \eqref{InequalitySupport}, $\vert x\vert\leq r_i<1/2$ for any $x\in \text{supp}(v_i)$. Hence, for any $x\in\text{supp}(v_i)$ and $y\in B_1(x)\backslash B_{1/2}(0)$, 
\[
\vert x - y \vert \geq \vert \vert x\vert - \vert y\vert \vert = \vert y\vert - \vert x\vert \geq \frac{1}{2}-r_i=:d_i>0.
\]
Then, for any $x\in\text{supp}(v_i)$,
\begin{align}
\kappa_{B_{1/2}(0)}(x) = c_N\int_{B_1(x)\backslash B_{1/2}(0)}\frac{1}{\vert x - y \vert^N} \; dy\leq c_N \int_{B_1(x)\backslash B_{1/2}(0)}\frac{1}{d_i^N}\leq \frac{c_N}{d_i^N}\int_{B_2(0)} dy = \frac{c_N2^N|B_1(0)|}{d_i^N}=:D_i,
\end{align}
because $B_1(x)\subset B_2(0)$ for $x\in \text{supp}(v_i)\subset B_{1/2}(0)$. Therefore,
\begin{equation}
\int_{B_{1/2}(0)}\kappa_{B_{1/2}(0)}v_i^2 \; dx = \int_{\text{supp}(v_i)}\kappa_{B_{1/2}(0)}v_i^2 \; dx \leq D_i \int_{\text{supp}(v_i)}v_i^2 \; dx = D_i\int_{B_{1/2}(0)}v_i^2 \; dx.
\end{equation}

Since $\text{supp}(\alpha_i)\subset U_i$, then also $\text{supp}(\alpha_i u)\subset U_i$. Hence, by the above estimate, the definition of $v_i$, \eqref{InequalityMetric}, and the definition of the integral on a chart on $\mathbb{S}^N$, we have that
\begin{align}
\int_{B_{1/2}(0)}\kappa_{B_{1/2}(0)} v_i^2 \: dx
&\leq D_i \int_{B_{1/2}(0)} v_i^2\; dx = D_i \int_{B_{1/2}(0)} (\alpha_i u)^2\circ h_i^{-1}(x)\; dx \nonumber\leq D_iC_0 \int_{B_{1/2}(0)} (\alpha_i u)^2\circ h_i^{-1}(x)\sqrt{\vert g\vert}\circ h_i^{-1}\; dx \nonumber\\
&= D_iC_0\int_{U_i}(\alpha_iu)^2 dV_g = D_iC_0\int_{\mathbb{S}^N}(\alpha_iu)^2 dV_g\leq D_0C_0\int_{\mathbb{S}^N}(\alpha_iu)^2 dV_g,
\label{EstimateL^2PartLogNorm}
\end{align}
where the constant $D_0:=\max\{D_i\;:\;i=1,\ldots,\ell\}>0$ depends on $N$,  on the chart $\{U_i,h_i\}_{i=1}^\ell$, and on the partition of the unity $\{\alpha_i\}_{i=1}^\ell$, but does not depend on $u\in\mathcal{C}^\infty(\mathbb{S}^N)$.

Next, we estimate the seminorm given by the kernel $K$. On the one hand, for any $x,y\in B_{1/2}(0)$ we have that $x-y\in B_1(0)$ and
\[
K(x-y) = c_N\frac{1_{B_1(0)}(x-y)}{\vert x - y \vert^N} = \frac{c_N}{\vert x - y \vert^N}.
\]
 Using this, \eqref{InequalityLipschitz}, \eqref{InequalityMetric}, the definition of $v_i$, and the definition of integral in the chart $U_i$, we have that
\begin{align}
&\int_{B_{1/2}(0)}\int_{B_{1/2}(0)} (v_i(x)-v_i(y))^2 K(x-y) \; dx dy \\
&= c_N\int_{B_{1/2}(0)} \int_{B_{1/2}(0)} \frac{(v_i(x)-v_i(y))^2}{\vert x - y\vert^N}  \; dx dy 
\leq c_N C_1^N \int_{B_{1/2}(0)} \int_{B_{1/2}(0)} \frac{(v_i(x)-v_i(y))^2}{\vert h_i^{-1}(x) - h_i^{-1}(y)\vert^N}  \; dx dy \nonumber\\
&= c_N C_1^N \int_{B_{1/2}(0)} \int_{B_{1/2}(0)} \frac{\left((\alpha_i u)\circ h_i^{-1}(x)-(\alpha_i u)\circ h_i^{-1}(y)\right)^2}{\vert h_i^{-1}(x) - h_i^{-1}(y)\vert^N}  \; dx dy\nonumber\\
&\leq  c_NC_0^2 C_1^N \int_{B_{1/2}(0)} \int_{B_{1/2}(0)} \frac{\left((\alpha_i u)\circ h_i^{-1}(x)-(\alpha_i u)\circ h_i^{-1}(y)\right)^2}{\vert h_i^{-1}(x) - h_i^{-1}(y)\vert^N} \sqrt{\vert g\vert}\circ h_i^{-1}(x)\sqrt{\vert g\vert}\circ h_i^{-1}(y)  \; dx dy\nonumber\\
&=  c_NC_0^2 C_1^N \int_{U_i} \int_{U_i} \frac{\left((\alpha_i u)(z)-(\alpha_i u)(\zeta)\right)^2}{\vert z - \zeta \vert^N}\; dV_g(z) dV_g(\zeta) 
\leq   c_NC_0^2 C_1^N \int_{\mathbb{S}^N} \int_{\mathbb{S}^N} \frac{\left((\alpha_i u)(z)-(\alpha_i u)(\zeta)\right)^2}{\vert z - \zeta \vert^N}\; dV_g(z) dV_g(\zeta).\label{EstimateSeminormLogNorm}
\end{align}
The constant $c_NC_0^2C_1^{N}$ only depends on $N$ and on the chart $\{U_i,h_i\}_{i=1}^\ell$. 

Hence, by the expression of the norm in $\mathbb{H}(B_{1/2}(0))$ given in \eqref{NormLogOmega}, and the estimates \eqref{EstimateL^2PartLogNorm} and \eqref{EstimateSeminormLogNorm}, and as $\alpha_iu\in\mathcal{C}^\infty(\mathbb{S}^N)\subset \mathbb{H}(\mathbb{S}^N)$, we have that
\[
\begin{split}
\Vert v_i\Vert^2_{\mathbb{H}(B_{1/2}(0))}& \leq  \frac{c_NC_0^2C_1^N}{2} \int_{\mathbb{S}^N} \int_{\mathbb{S}^N} \frac{\left((\alpha_i u)(z)-(\alpha_i u)(\zeta)\right)^2}{\vert z - \zeta \vert^N}\; dV_g(z) dV_g(\zeta) + D_0C_0\int_{\mathbb{S}^N}(\alpha_iu)^2\; dV_g\\
& \leq C \left[ \frac{c_N}{2} \int_{\mathbb{S}^N} \int_{\mathbb{S}^N} \frac{\left((\alpha_i u)(z)-(\alpha_i u)(\zeta)\right)^2}{\vert z - \zeta \vert^N}\; dV_g(z) dV_g(\zeta) + \kappa \int_{\mathbb{S}^N}(\alpha_iu)^2\; dV_g\right]= C\Vert \alpha_iu\Vert_{\mathbb{H}(\mathbb{S}^N)}^2,
\end{split}
\]
where $C:=\max\{C_0^2C_1^{N},D_0C_0 \kappa^{-1}\}>0$ does not depend on $u\in\mathcal{C}^\infty(\mathbb{S}^N).$
\end{proof}

\begin{lemma}\label{Lemma:BoundL^2Norm}
For $C_0>0$ as in \eqref{InequalityMetric}, $\Vert \alpha_iu \Vert_{L_g^2(\mathbb{S}^N)} \leq C_0^{1/2}\Vert v_i \Vert_{L^2(B_{1/2}(0))}$ for all $u\in\mathcal{C}^\infty(\mathbb{S}^N).$
\end{lemma}
\begin{proof}
By \eqref{InequalityMetric},
\[
\begin{split}
\Vert \alpha_iu \Vert_{L_g^2(\mathbb{S}^N)}^2 & = \int_{\mathbb{S}^N}\alpha_i^2 u^2 \:dV_g = \int_{U_i}\alpha_i^2 u^2 \:dV_g= \int_{B_{1/2}(0)}(\alpha_i u)^2\circ h_i^{-1}(x)\sqrt{\vert g\vert}\circ h_i^{-1}(x) dx\\
&\leq C_0 \int_{B_{1/2}(0)}(\alpha_i u)^2\circ h_i^{-1}(x) dx
\leq C_0 \int_{B_{1/2}(0)}v_i^2 dx= C_0 \Vert v_i \Vert_{L^2(B_{1/2}(0))}^2.
\end{split}
\]
\end{proof}

\begin{lemma}\label{Lemma:LogNormConstants}
For $\alpha\in\mathcal{C}^\infty(\mathbb{S}^N)$, there is $D>0$ such that $\Vert \alpha u\Vert^2_{\mathbb{H}(\mathbb{S}^N)}\leq D\Vert u\Vert^2_{\mathbb{H}(\mathbb{S}^N)}$ for $u\in\mathbb{H}(\mathbb{S}^N)$.
\end{lemma}
\begin{proof}
Recall the notation at the beginning of Section \ref{prop:sec} and recall inequality \eqref{ComparisonGeodesicCordalDistance} between the geodesic distance $d_g$ in $\mathbb{S}^N$ and the Euclidean distance. By the density of $\mathcal{C}^\infty(\mathbb{S}^N)$ in $\mathbb{H}(\mathbb{S}^N)$, it suffices to show the inequality for $u\in \mathcal{C}^\infty(\mathbb{S}^N)$.   As $\alpha$ is smooth, by compactness of $\mathbb{S}^N$, its differential is globally bounded, implying that $\alpha$ is globally Lipschitz continuous in $\mathbb{S^N}$, and, therefore, there is $D_2>0$ so that
\begin{equation}\label{SmoothFunctionsLipschitz}
\vert \alpha(z) - \alpha(\zeta) \vert \leq D_2 d_g(z,\zeta)\qquad\text{ for all }z,\zeta\in\mathbb{S}^N.
\end{equation}

Next, we use the notation from \eqref{e1} and \eqref{e2}. Let $\eta:\mathbb{S}^N\rightarrow\mathbb{R}$ be given by $\eta(z):=\int_{\mathbb{S}^N}d_g(z,\zeta)^{2-N}\;dV_g(\zeta).$ By \eqref{e1} and \eqref{e2}, we have that 
\begin{align}\label{eta:bdd}
\text{$\eta$ is uniformly bounded in $\mathbb{S}^N$.}
\end{align}Indeed, for $N=1,2$, this is obvious. For $N\geq 3$, as $D_3$ and $D_4$ do not depend on $z\in\mathbb{S}^N,$ we have that
\begin{align*}
\eta(z) &= \int_{B_g(z,\rho)} d_g(z,\zeta)^{2-N}dV_g(\zeta) + \int_{\mathbb{S}^N\backslash B_g(z,\rho)} d_g(z,\zeta)^{2-N}dV_g(\zeta)
\leq D_4^{N-2} \int_{B_g(z,\rho)} \vert \exp_z^{-1}(\zeta)\vert^{2-N}dV_g(\zeta) + \vert \mathbb{S}^N\vert\rho^{2-N}\\
&=D_4^{N-2} \int_{B_\rho(0)} \vert \exp_z^{-1}\circ h_{z}^{-1}(y)\vert^{2-N}\sqrt{\vert g\vert}\circ h_z^{-1}(y)\;dy + \vert \mathbb{S}^N\vert\rho^{2-N}
\leq D_4^{N-2}D_3 \int_{B_\rho(0)} \vert y\vert^{2-N}\;dy + \vert \mathbb{S}^N\vert\rho^{2-N}\\
& = D_4^{N-2}D_3 N|B_1(0)| \int_0^\rho t^{2-N}t^{N-1}\;dy + \vert \mathbb{S}^N\vert\rho^{2-N}
=D_4^{N-2}D_3 N|B_1(0)|\frac{\rho^2}{2}  + \vert \mathbb{S}^N\vert\rho^{2-N}=:\widehat{D},
\end{align*}
where $\widehat{D}>0$ does not depend on $z\in\mathbb{S}^N$. Let $D_1$ be as in \eqref{ComparisonGeodesicCordalDistance}.  Then, by \eqref{SmoothFunctionsLipschitz},
\begin{align*}
&\Vert \alpha u\Vert^2_{\mathbb{H}(\mathbb{S}^N)}= \frac{c_N}{2}\int_{\mathbb{S}^N}\int_{\mathbb{S}^N}\frac{(\alpha(z)u(z) - \alpha(\zeta)u(\zeta))^2}{\vert z - \zeta\vert^{N}}
+ \kappa\int_{\mathbb{S}^N} \alpha^2u^2\\
&\leq \frac{c_N}{2}\int_{\mathbb{S}^N}\int_{\mathbb{S}^N}\frac{\left(\vert\alpha(\zeta)\vert \vert u(z) - u(\zeta)\vert + \vert \alpha(z) - \alpha(\zeta)\vert \vert u(z)\vert\right)^2}{\vert z - \zeta\vert^{N}}
+ \kappa\|\alpha\|_{L_g^\infty(\S)}^2\int_{\mathbb{S}^N} u^2\\
&\leq 2 c_N\int_{\mathbb{S}^N}\int_{\mathbb{S}^N}\frac{\vert\alpha(\zeta)\vert^2 \vert u(z) - u(\zeta)\vert^2}{\vert z - \zeta\vert^{N}} + 2 c_N\int_{\mathbb{S}^N}\int_{\mathbb{S}^N}\frac{ \vert \alpha(z) - \alpha(\zeta)\vert^2 \vert u(z)\vert^2}{\vert z - \zeta\vert^{N}} 
+ \kappa\|\alpha\|_{L_g^\infty(\S)}^2\int_{\mathbb{S}^N} u^2\\
&\leq 2\|\alpha\|_{L_g^\infty(\S)}^2 c_N\int_{\mathbb{S}^N}\int_{\mathbb{S}^N}\frac{ \vert u(z) - u(\zeta)\vert^2}{\vert z - \zeta\vert^{N}} +2 c_N D_1^ND^2_2\int_{\mathbb{S}^N}\int_{\mathbb{S}^N} d_g(z,\zeta)^{2-N} \vert u(z)\vert^2
+ \kappa\|\alpha\|_{L_g^\infty(\S)}^2\int_{\mathbb{S}^N} u^2\\
&= 2\|\alpha\|_{L_g^\infty(\S)}^2 c_N \int_{\mathbb{S}^N}\int_{\mathbb{S}^N}\frac{ \vert u(z) - u(\zeta)\vert^2}{\vert z - \zeta\vert^{N}} +2 c_N D_1^ND^2_2\int_{\mathbb{S}^N} \vert u(z)\vert^2\eta(z)
+ \kappa\|\alpha\|_{L_g^\infty(\S)}^2\int_{\mathbb{S}^N} u^2\\
&\leq 2\|\alpha\|_{L_g^\infty(\S)}^2c_N\int_{\mathbb{S}^N}\int_{\mathbb{S}^N}\frac{ \vert u(z) - u(\zeta)\vert^2}{\vert z - \zeta\vert^{N}} +2 c_N D_1^ND^2_2\widehat{D}\int_{\mathbb{S}^N} \vert u(z)\vert^2 
+ \kappa\|\alpha\|_{L_g^\infty(\S)}^2\int_{\mathbb{S}^N} u^2 \leq D\Vert u\Vert^2_{\mathbb{H}(\mathbb{S}^N)},
\end{align*}
where $D:=\max\{4\|\alpha\|_{L_g^\infty(\S)}^2, 2 \kappa^{-1}c_ND_1^ND_2^2\widehat{D} + \|\alpha\|_{L_g^\infty(\S)}^2\}>0$ does not depend on $u\in\mathcal{C}^\infty(\mathbb{S}^N)$, as we wanted to prove.
\end{proof}

We are ready to show Proposition~\ref{embed:prop}.
\begin{proof}[Proof of Proposition~\ref{embed:prop}]
Let $(u_k)$ be a bounded sequence in $\mathbb{H}(\mathbb{S}^N)$.
By the density of $\mathcal{C}^\infty(\mathbb{S}^N)$ in $\mathbb{H}(\mathbb{S}^N)$ and in $L_g^2(\mathbb{S}^N)$, we can assume, without loss of generality, that $u_k\in \mathcal{C}^\infty(\mathbb{S}^N)$ for each $k\in\mathbb{N}$. Consider the chart $\{U_i,h_i\}_{i=1}^\ell$ and the partition of the unity $\{\alpha_i\}_{i=1}^\ell$ given at the beginning of this section and, for each $k\in\mathbb{N}$ and each $i\in\{1,\ldots,\ell\}$, define
\[
v_{k,i}:\mathbb{R}^N\rightarrow\mathbb{R}, \quad v_{k,i}(x):= \begin{cases}
 (\alpha_i u_k)\circ h_i^{-1}(x), & x \in B_{1/2}(0),\\
 0, & x\notin B_{1/2}(0).
 \end{cases}
\]
By Lemmas~\ref{Lemma:EstimatesLogNorm} and~\ref{Lemma:LogNormConstants}, the sequences $(v_{k,i})$ are bounded in $\mathbb{H}(B_{1/2}(0))$. Indeed, for each $i\in\{1,\ldots,\ell\},$
\[
\Vert v_{k,i}\Vert_{\mathbb{H}(B_{1/2}(0))}\leq C\Vert\alpha_i u_k \Vert_{\mathbb{H}(\mathbb{S}^N)}\leq CD \Vert u_k \Vert_{\mathbb{H}(\mathbb{S}^N)}\leq CDM,
\]
where $C,D>0$ are given by the Lemmas~\ref{Lemma:EstimatesLogNorm} and~\ref{Lemma:LogNormConstants}, and $M>0$ is a bound for $(u_k)$ in $\mathbb{H}(\mathbb{S}^N)$. Hence, by Lemma~\ref{Lemma:CompactnessOmega}, for each $i\in\{1,\ldots,\ell\}$, there is a subsequence of $(v_{k,i})$ that converges in $L^2(B_{1/2}(0))$. We can extract a uniform subsequence such that the convergence holds true for every $i=1,\ldots,\ell$. Denote these sequences by $(v_{k_j,i})$. Then, these are Cauchy sequences and given $\varepsilon>0$, there exists $J\in\mathbb{N}$ such that, for any $m,n\geq J$,
\[
\Vert v_{k_m,i}-v_{k_n,i}\Vert_{L^2(B_{1/2}(0))}<\frac{\varepsilon}{C_0^{1/2}\ell}.
\]
Hence, as $\sum_{i=1}^\ell\alpha_i=1$, for any $m,n\geq J$, Lemma~\ref{Lemma:BoundL^2Norm} yields that
\[
\begin{split}
\Vert u_{k_m} - u_{k_n}\Vert_{L_g^2(\mathbb{S}^N)} &= \left\Vert \sum_{i=1}^\ell \alpha_i(u_{k_m} - u_{k_n})\right\Vert_{L_g^2(\mathbb{S}^N)}\leq \sum_{i=1}^\ell \Vert \alpha_i(u_{k_m} - u_{k_n})\Vert_{L_g^2(\mathbb{S}^N)}
\leq C_0^{1/2} \sum_{i=1}^\ell \Vert (v_{k_m,i} - v_{k_n,i})\Vert_{L_g^2(B_{1/2}(0))}< \varepsilon,
\end{split}
\]
proving that $(u_{k,j})_{j}$ is a Cauchy sequence in $L^2_g(\mathbb{S}^N)$ and, hence, it converges in this space.
\end{proof}

\subsection{An isometric isomorphism}

Using $\iota$ we can define a norm in the following way,
\begin{align}\label{Dnormdef}
\|v\|_{\mathbb{D}(\mathbb{R}^N)}:=\|u\|_{\H},\qquad \text{ if $v=\iota(u)$ for some $u\in C^\infty(\S)$}.
\end{align}

Now, consider the completion of $\iota(C^\infty(\S))$ with respect to the norm $\|\cdot\|_{\mathbb{D}(\mathbb{R}^N)},$ namely, 
\begin{align}\label{Ddef}
    \mathbb{D}(\rn)
    := \overline{\iota(C^\infty(\S))}^{\|\cdot\|_{\mathbb{D}(\mathbb{R}^N)}}
    = \iota\left(\overline{C^\infty(\S)}^{\|\cdot\|_{\H}}\right)
    =\iota(\H).
\end{align}
The norm $\|\cdot\|_{\D}$ (defined originally on $\iota(C^\infty(\S))$) has a unique extension to $\D$.

This construction mimics what happens in the fractional case, where $\iota_s(u):=\phi^{\frac{N-2s}{2}}u\circ \sigma^{-1}$ gives an isometric isomorphism between the fractional Sobolev space on the sphere, $H^s(\S),$ and the homogeneous fractional Sobolev space in $\rn$, $D^s(\rn)$; see, for instance, \cite[Proposition 2.4]{CFS25}.  See also \cite{D86,CSS21,CFS21} for similar relationships in the local setting $s\in \mathbb N$. 

Since $\H$ is a Hilbert space, we can define the inner product 
\begin{align}\label{scaldef}
\langle v_1,v_2\rangle_{\mathbb{D}(\mathbb{R}^N)}:=\langle u_1,u_2\rangle_{\H}
\end{align}
for $v_1,v_2\in \D$ so that $v_i=\iota(u_i)$ and $u_i\in \H$ for $i=1,2$.

\begin{theorem}\label{Ddefthm}
    The space $\mathbb{D}(\mathbb{R}^N)$ is a Hilbert space and the embedding $\mathbb{D}(\mathbb{R}^N)\hookrightarrow L^2(\mathbb{R}^N)$ is compact.  The map $\iota: L_g^2(\mathbb{S}^N)\rightarrow L^2(\mathbb{R}^N)$ and $\iota: \mathbb{H}(\mathbb{S}^N)\rightarrow \mathbb{D}(\mathbb{R}^N)$ (given by \eqref{iotadef}) are isometric isomorphisms. Moreover, 
\begin{align}\label{Dsnorm}
\|v\|_{\mathbb{D}(\mathbb{R}^N)}=\left(\cE_L(v,v)+\int_{\rn}v^2\ln \phi^{-2}\, dx
+(\kappa-A_N) \int_{\rn}v^2\,dx
\right)^\frac{1}{2}\qquad \text{for $v\in C^\infty_c(\rn)$},
\end{align}
where $\kappa$ is the same fixed constant as in \eqref{NormLogSobolevSphere} and \eqref{kappahyp}, and 
\begin{align}\label{Dsprod}
\langle v_1,v_2\rangle_{\mathbb{D}(\mathbb{R}^N)}= \cE_L(v_1,v_2)+\int_{\rn}v_1 v_2(\kappa-A_N+\ln \phi^{-2})\qquad \text{for $v_1,v_2
\in C^\infty_c(\rn)$.}
\end{align}
\end{theorem}
\begin{proof}
By \eqref{Dnormdef}, $\iota: C_c^\infty(\S\backslash \{-e_{N+1}\})\to \iota(C_c^\infty(\S\backslash \{-e_{N+1}\}))$ is an isometric isomorphism.  Since $C_c^\infty(\S\backslash \{-e_{N+1}\})$ is dense in $\mathbb{H}(\mathbb{S}^N)$, by Lemma \ref{density:lem2}, we know that there is a unique isometric isomorphism that extends $\iota$ (which we denote again by $\iota$) from $\mathbb{H}(\mathbb{S}^N)$ to $\mathbb{D}(\mathbb{R}^N)$.  Let $u_1,u_2\in C_c^\infty(\S\backslash \{-e_{N+1}\})$ and let $v_i=\iota(u_i)\in C^\infty_c(\rn)$ for $i=1,2$. Then, by Proposition~\ref{Prop:ConformalWithEuclidean}, \eqref{AN:form}, and \eqref{scaldef},
\begin{align*}
\langle v_1,v_2\rangle_{\mathbb{D}(\mathbb{R}^N)}&=
\langle u_1, u_2\rangle_{\mathbb{H}(\mathbb{S}^N)}=\int_{\mathbb{S}^N}u_1\mathscr{P}_g^{\log}u_2 \; dV_g 
+(\kappa-A_N)\int_{\mathbb{S}^N}u_1 u_2 \; dV_g \\
&= \int_{\mathbb{R}^N}u_1\circ\sigma^{-1}(x) \mathscr{P}_g^{\log}u_2\circ\sigma^{-1}(x)\phi^N(x) \; dx
+(\kappa-A_N)\int_{\rn}u_1\circ\sigma^{-1}(x) u_2\circ\sigma^{-1}(x) \phi^N(x)\; dx 
\\
&=  \int_{\mathbb{R}^N}v_1(x) \left[ \phi^{-\frac{N}{2}}(x)L_\Delta(\phi^{\frac{N}{2}}(x)u_2\circ\sigma^{-1}(x)) - 2u_2\circ\sigma^{-1}(x)\ln\phi(x) \right]\phi^{\frac{N}{2}}(x) \; dx
+(\kappa-A_N)\int_{\rn}v_1v_2\; dx
\\
&=  \int_{\mathbb{R}^N}v_1 \left[L_\Delta v_2 - v_2\ln(\phi^2) \right] \; dx
+(\kappa-A_N)\int_{\rn}v_1v_2\; dx.
\end{align*}
Hence, \eqref{Dsprod} holds because, by \cite{CW19}, $\cE_L(v_1,v_2)=\int_{\rn} v_1 L_\Delta v_2\, dx$.  The compactness of the embedding $\mathbb{D}(\mathbb{R}^N)\hookrightarrow L^2(\mathbb{R}^N)$ follows from Proposition~\ref{embed:prop}. That $\iota: L_g^2(\mathbb{S}^N)\rightarrow L^2(\mathbb{R}^N)$ is an isometric isomorphism can be argued similarly, and it is also observed in \cite[Sec.~2.1]{FKT20}.
\end{proof}

\begin{remark}\label{for:rmk}
\emph{
Note that formulas \eqref{Dsnorm} and \eqref{Dsprod} are shown to hold under additional regularity assumptions. However, after Theorems \ref{thmD} and \ref{Dlogprop2} are shown, it follows that these representations also hold for any $v\in \D$ and that each term is well defined.
}
\end{remark}

\begin{remark}\label{exp:rmk}
\emph{
Another way of looking at Theorem~\ref{Ddefthm} is the following: 
we know that, for $u\in  C_c^\infty(\mathbb{S}^N\backslash\{-e_{N+1}\})$,
\begin{align*}
v_s:=\iota_s(u)=\phi^{\frac{N-2s}{2}}u(\sigma^{-1}(x))
=(1-s\ln(\phi)+o(s))\phi^{\frac{N}{2}}u(\sigma^{-1}(x))
=(1-s\ln(\phi)+o(s))v,\qquad v:=\iota(u),
\end{align*}
satisfies that 
\begin{align*}
\|u\|^2_{H_g^s(\mathbb{S}^N)}=\|v_s\|^2_{D^s(\rn)}
:= \frac{c_{N,s}}{2}\int_{\rn}\int_{\rn}\frac{(v_s(x)-v_s(y))^2}{\vert x - y\vert^{N+2s}}\, dx\, dy=\int_{\rn}|\widehat{v_s}(\xi)|^2|\xi|^{2s}\, d\xi,
\end{align*}
see, for instance, \cite[Proposition 2.4]{CFS25}. Moreover, by Theorem~\ref{main:thm},
$\mathscr{P}^s_g u=u+s\mathscr{P}^{log}_g u+o(s)$, and we also have that $|\xi|^{2s}=1+s\ln |\xi|^2+o(s)$. Hence,
\begin{align*}
\int_{\mathbb{S}^N}&u (u+s\mathscr{P}^{log}_g u+o(s))\; dV_g 
=\int_{\mathbb{S}^N}u \mathscr{P}^s_g u\; dV_g\\
&=\frac{c_{N,s}}{2}\int_{\mathbb{S}^N}\int_{\mathbb{S}^N} \frac{(u(z)-u(\zeta))^2}{\vert z - \zeta\vert^{N+2s}} \;dV_g(z)\;dV_g(\zeta) + A_{N,s}\int_{\mathbb{S}^N}u^2\; dV_g \\
&=\|u\|^2_{H_g^s(\mathbb{S}^N)}=\|v_s\|^2_{D^s(\rn)}
=\int_{\rn}|\widehat{v_s}(\xi)|^2(1+s\ln |\xi|^2+o(s))\, d\xi\\
&=\int_{\rn}(1-s\ln(\phi)+o(s))^2|v|^2+s\int_{\rn}|\widehat{v_s}(\xi)|^2(\ln |\xi|^2+o(1))\, d\xi\\
&=\int_{\rn}\Big(1-2(s\ln(\phi)+o(s))+(s\ln(\phi)+o(s))^2\Big)|v|^2+s\int_{\rn}|\widehat{v_s}(\xi)|^2(\ln |\xi|^2+o(1))\, d\xi,
\end{align*}
which yields that 
\begin{align*}
\|u\|^2_{\H}
&=\int_{\S}u\mathscr{P}^{log}_g u\, dV_g + (\kappa-A_N) \int_{\mathbb{S}^N} u^2 \, dV_g
=\int_{\rn}-2\ln(\phi)v^2\, dx+\int_{\rn}|\widehat{v}(\xi)|^2\ln |\xi|^2\, d\xi+ (\kappa-A_N)\int_{\rn} v^2 \, dx\\
&=\int_{\rn}(\ln(\phi^{-2})v+L_\Delta v)v\, dx + (\kappa-A_N) \int_{\rn} v^2 \, dx=\|v\|^2_{\D}.
\end{align*}
}
\end{remark}

\subsection{An equivalence}\label{disc:sec}

In this section, we show the following result relating  $\D$ (defined in \eqref{Ddef}) and $D^{log}(\rn)$ (defined in \eqref{Dlogdef}).  Recall that 
\begin{align*}
\|v\|:=\left(\cE(v,v)+\int_{\rn}v^2\ln(e+|x|^2)\, dx\right)^\frac{1}{2}.
\end{align*}
\begin{theorem}\label{thmD}
There is $C>1$ such that 
\begin{align}\label{equiv}
    C^{-1}\|v\|\leq \|v\|_{\D}\leq C\|v\|\qquad \text{for all $v\in C_c^\infty(\rn)$.}
\end{align}
As a consequence, there is an isomorphism between $\D$ and $D^{log}(\rn)$.
\end{theorem}

Next, let us recall Pitt's inequality. For this, we specify that our definition of Fourier transform is
$$
\hat{v}(\xi) := (2\pi)^{-N/2} \int_{\mathbb{R}^N} v(x)e^{-i\xi \cdot x} \,dx, \quad \xi \in \mathbb{R}^N,\quad v\in C^\infty_c(\rn).
$$
The following is Pitt's inequality with a sharp constant. This is shown in \cite[Theorem 1]{B95}; we have adjusted the constants to match our definition of Fourier transform. 

\begin{proposition}[Pitt's inequality]\label{Pitt:prop}
Let $v\in C_c^{0,1}(\rn)$, then
\begin{align}\label{Pittineq}
\cE_L(v,v)+\int_{\mathbb{R}^N}\ln(|x|^{2})|v(x)|^2 \,dx 
\geq a_N\|v\|_{L^2(\rn)}^2,\qquad a_N:=2 \psi\left(\tfrac{N}{4}\right)+2\ln (2).
\end{align}
\end{proposition}

\begin{remark}
\emph
{
Proposition~\ref{Pitt:prop} is shown in \cite[Theorem 1]{B95} for smooth functions, however a standard density argument shows that it also holds for functions in $C_c^{0,1}(\rn)$ (and even less regular functions). In fact, Proposition~\ref{Pitt:prop} can also be shown with a strategy as in Remark~\ref{exp:rmk} starting from the fractional Hardy inequality. To be more precise, the fractional Hardy inequality states that, for $ v\in H^s(\rn)$,
\begin{equation} \label{eq:2.1}
\int_{\mathbb{R}^N} |x|^{-2s}|v(x)|^2 \,dx \leq C_{s,N}^{-1} \int_{\mathbb{R}^N} |\xi|^{2s}|\hat{v}(\xi)|^2 \,d\xi, 
\qquad
C_{s,N} := 2^{2s} \frac{\Gamma^2((N+2s)/4)}{\Gamma^2((N-2s)/4)},
\end{equation}
where $0 < 2s < N$, see \cite{H77,Y99,B95}, or \cite[Section 2.1]{FS08}. Let $v\in C_c^{0,1}(\rn)$. Using that 
\begin{align*}
|\xi|^{2s}=\sum_{k=0}^\infty \frac{(2s \ln |\xi|)^k}{k!},\qquad
|x|^{-2s}= \sum_{k=0}^\infty \frac{(-2s \ln |x|)^k}{k!},\qquad C_{s,N}^{-1}=1-s a_N+o(s)
\end{align*}
 as $s\to 0,$ we have that 
\begin{align*}
\int_{\mathbb{R}^N}&|v(x)|^2+\left(-2s\ln|x|+\sum_{k=2}^\infty \frac{(-2s \ln |x|)^k}{k!}\right)|v(x)|^2 \,dx=\int_{\mathbb{R}^N} |x|^{-2s}|v(x)|^2 \,dx \\
&\leq C_{s,N}^{-1} \int_{\mathbb{R}^N} |\xi|^{2s}|\widehat{v}(\xi)|^2 \,d\xi=(1-s a_N+o(s))\int_{\mathbb{R}^N}\left(1+s\ln|\xi|^2+\sum_{k=2}^\infty \frac{(2s \ln |\xi|)^k}{k!}\right)|\widehat{v}(\xi)|^2 \,d\xi.
\end{align*}
This yields
\begin{align*}
\int_{\mathbb{R}^N}\ln(|x|^{-2})|v(x)|^2 \,dx 
\leq \int_{\mathbb{R}^N}(-a_N+\ln|\xi|^2)|\hat{v}(\xi)|^2 \,d\xi
=-a_N\|v\|_{L^2(\rn)}^2+\cE_L(v,v),    
\end{align*}
as claimed.
}
\end{remark}

\begin{remark}[The inner product in $\D$ is positive definite]\label{posdef:rmk}
\emph{
Using Pitt's inequality, one can show directly that the right-hand side of \eqref{Dsprod} is positive definite. 
Let $v\in C_c^\infty(\rn)$. Using that $\ln(1 + |x|^2) = \ln |x|^2 + \ln\left(1 + \frac{1}{|x|^2}\right)$ for $x \ne 0,$ Pitt's inequality \eqref{Pittineq}, and Theorem~\ref{Ddefthm}, we obtain that
\begin{align}
\langle v,v\rangle_{\D}&=
\|v\|^2_{\mathbb{D}(\mathbb{R}^N)}=\cE_L(v,v)+2\int_{\rn}v^2\ln\left(\frac{1+|x|^2}{2}\right)\, dx +(\kappa-A_N)\int_{\rn}v^2\,dx\notag\\
&=\cE_L(v,v)
+2\int_{\rn}v^2\left(\ln |x|^2 + \ln\left(1 + \frac{1}{|x|^2}\right)\right)\, dx
+(\kappa-A_N-2\ln 2) \int_{\rn}v^2\,dx\notag\\
&=\cE_L(v,v)
+2\int_{\rn}v^2\ln |x|^2\, dx
+2\int_{\rn}v^2\ln\left(1 + \frac{1}{|x|^2}\right)\, dx
+(\kappa-A_N-2\ln 2) \int_{\rn}v^2\,dx\notag\\
&=\left(\cE_L(v,v)
+\int_{\rn}v^2\ln |x|^2\, dx\right)
+\int_{\rn}v^2\ln |x|^2\, dx
+2\int_{\rn}v^2\ln\left(1 + \frac{1}{|x|^2}\right)\, dx
+(\kappa-A_N-2\ln 2) \int_{\rn}v^2\,dx\notag\\
&\geq a_N\|v\|_{L^2(\rn)}^2
+\left(\int_{\rn}v^2\ln |x|^2\, dx
+2\int_{\rn}v^2\ln\left(1 + \frac{1}{|x|^2}\right)\, dx\right)
+(\kappa-A_N-2\ln 2)\|v\|_{L^2(\rn)}^2\notag\\
&=2\int_{\rn}v^2
\ln\left(|x| + \frac{1}{|x|}\right)
\, dx
+\left(2\psi(\tfrac{N}{4})+\kappa-A_N\right) \|v\|_{L^2(\rn)}^2\label{Pittc}
\end{align}
(recall that $a_N:=2 \psi\left(\tfrac{N}{4}\right)+2\ln (2)$ and $A_N:= 2 \psi\left(\tfrac{N}{2}\right)$). Observe that $2\ln\left(|x| + \frac{1}{|x|}\right)\geq 2\ln 2$ for $|x|\neq 0$ (the minimum is achieved at $|x|=1$). Hence,
\begin{align*}
\langle v,v\rangle_{\D}=\|v\|^2_{\mathbb{D}(\mathbb{R}^N)}\geq 
(2\ln 2+2(\psi(\tfrac{N}{4})
-\psi\left(\tfrac{N}{2}\right)
)+\kappa)
\|v\|_{L^2(\rn)}^2
\geq 0,
\end{align*}
since $\kappa>2(\psi(\tfrac{N}{2})-\psi(\tfrac{N}{4}))$, by assumption, see \eqref{kappahyp}.
}
\end{remark}

\begin{remark}[Finiteness of the $L^2$-logarithmic term]\label{l2:rmk}
\emph{
Let $v\in \D$. By Lemma~\ref{density:lem2} and \eqref{Cc}, there are $v_n\in C_c^\infty(\rn)$ such that $v_n\to v$ in $\D$. Observe that $\ln\left(1+|x|^2\right) \leq 2\ln\left(|x|+\frac{1}{|x|}\right)$ for $|x|\neq 0$.  Hence, we have, by \eqref{Pittc} and Fatou's lemma, that 
\begin{align*}
\int_{\rn}|v|^2 \ln (1+|x|^2) \, dx
&\leq
\liminf_{n\to\infty}\int_{\rn}|v_n|^2 \ln (1+|x|^2) \, dx\\
&\leq 
\liminf_{n\to\infty}\int_{\rn}|v_n|^2 \ln\left[\left(|x| + \frac{1}{|x|}\right)^2\right] \, dx
\leq \liminf_{n\to\infty}\|v_n\|^2_{\D}
=\|v\|^2_{\D}.
\end{align*}
Together with the convergence in $L^2(\rn)$, we have that $\left|\int_{\rn}|v|^2 \ln \phi^{-2} \, dx\right|<\infty$ for all $v\in \D.$
}
\end{remark}

For the following auxiliary computation, given $r>0$ and $z\in\mathbb{S}^N$, define
\[
U_r(z):=\{\zeta\in\mathbb{S}^N\;:\; \vert \sigma(z)-\sigma(\zeta)\vert\leq r\},
\]
and denote by $B_r(x)$ the ball in $\mathbb{R}^N$ centered at $x$ and radius $r>0$.
\begin{lemma}\label{h:lem}
The function $h:\S\to \r$ given by
 \begin{align*}
 h(z):=\int_{U_1(z)}\frac{
\left|\frac{\phi^{N/2}(\sigma(z))}{\phi^{N/2}(\sigma(\zeta))}-1\right|}{|z-\zeta|^N}\, dV_g(\zeta)
\end{align*}
is uniformly bounded in $\S$.
\end{lemma}
\begin{proof}
Let $z,\zeta\in \S$, $y=\sigma(\zeta)$, and $x=\sigma(z).$ By \eqref{sigma}, 
\begin{align*}
|x|^2 = \frac{|z'|^2}{(1+z_{N+1})^2} = \frac{1-z_{N+1}^2}{(1+z_{N+1})^2} = \frac{1-z_{N+1}}{1+z_{N+1}}.
\end{align*}
Hence, $1+|x|^2 = \frac{2}{1+z_{N+1}}$, i.e.
\begin{equation}\label{eq:4}
\phi(x)=1+z_{N+1}.
\end{equation}
Furthermore,
\begin{align*}
|x-y|^2
&= |x|^2+|y|^2-2x\cdot y
=\frac{1-z_{N+1}}{1+z_{N+1}}+\frac{1-\zeta_{N+1}}{1+\zeta_{N+1}}-\frac{2z'\cdot\zeta'}{(1+z_{N+1})(1+\zeta_{N+1})}\notag\\
&= \frac{2-2z\cdot \zeta}{(1+z_{N+1})(1+\zeta_{N+1})}
= \frac{|z-\zeta|^2}{(1+z_{N+1})(1+\zeta_{N+1})}
\end{align*}
and therefore
\begin{align}\label{eq:5}
|x-y|^2\phi(x)\phi(y)=|z-\zeta|^2.
\end{align}
Hence, with a change of variables and \eqref{volumeform},
\begin{align*}
H(x)&:=h(\sigma^{-1}(x))
=\int_{B_1(x)}\frac{
\left|\frac{\phi^{N/2}(x)}{\phi^{N/2}(y)}-1\right|}{|x-y|^N \phi(x)^{N/2}\phi(y)^{N/2}}
\phi(y)^{N}
\, dy
=\phi(x)^{-N/2}\int_{B_1(x)}\frac{
|\phi^{N/2}(x)-\phi^{N/2}(y)|}{|x-y|^N}
\, dy\\
&=(1+|x|^2)^{N/2}\int_{B_1(x)}\frac{
\left|\frac{1}{(1+|x|^2)^{N/2}}-\frac{1}{(1+|y|^2)^{N/2}}\right|}{|x-y|^N}
\, dy
=\int_{B_1(0)}\frac{
|(1+|x|^2)^{N/2}-(1+|x+y|^2)^{N/2}|}{|y|^N(1+|x+y|^2)^{N/2}}
\, dy.
\end{align*}

Let $g(x):=(1+|x|^2)^{N/2},$ then
\begin{align*}
|g(x)-g(x+y)|&\leq |y|\int_0^1 |\nabla g(x+y-sy)|\, ds
\leq N|y|\int_0^1 (1+|x+y-sy|^2)^{\frac{N-2}{2}}|x+y-sy|\, ds.
\end{align*}
Therefore, 
\begin{align*}
H(x)\leq N
\int_{B_1(0)}\int_0^1\frac{
 (1+|x+y-sy|^2)^{\frac{N-2}{2}}|x+y-sy|}{|y|^{N-1}(1+|x+y|^2)^{N/2}}
\, ds\, dy.
\end{align*}
In the following, we use $C>0$ to denote possibly different constants depending only on $N$. 

Assume first that $N\geq 2$.  For $|x|<2$, clearly $H(x)\leq C$. Let $|x|>2$, then using the change of variables $y=|x|w$ ($dy = |x|^N dw$),
\begin{align*}
H(x)
&\leq 
N|x|^{N}
\int_{B_{|x|^{-1}}(0)}\int_0^1\frac{
 |x|^{N-2}(|x|^{-2}+|\frac{x}{|x|}+w-sw|^2)^{\frac{N-2}{2}}
 |x|||x|^{-1}+w-sw|}{|x|^{N-1}|w|^{N-1}  |x|^{N}(|x|^{-2}+|\frac{x}{|x|}+w|^2)^{N/2}}
\, ds\, dw\\
&=
N
\int_{B_{|x|^{-1}}(0)}\int_0^1\frac{
 (|x|^{-2}+|\frac{x}{|x|}+w-sw|^2)^{\frac{N-2}{2}}
 ||x|^{-1}+w-sw|}{|w|^{N-1} (|x|^{-2}+|\frac{x}{|x|}+w|^2)^{N/2}}
\, ds\, dw\\
&\leq C\int_{B_{|x|^{-1}}(0)}|w|^{1-N}\, dw
=C\int_0^{|x|^{-1}}\rho^{1-N}\rho^{N-1}\, d\rho=C|x|^{-1}\leq C,
\end{align*}
where we used that, since $x/|x|\in \partial B_1(0)$ and $w\in B_{|x|^{-1}}(0)$, we have that $|\frac{x}{|x|}+w|\geq 1-|w|\geq 1-|x|^{-1}>\frac{1}{2}$.

Dimension $N=1$ can be argued similarly. Indeed, note that, for $N=1$,
\begin{align*}
H(x)&
\leq 
\int_{-1}^1\int_0^1\frac{
 |x+y-sy|}{(1+|x+y-sy|^2)^{\frac{1}{2}}(1+|x+y|^2)^{1/2}}
\, ds\, dy
\leq \int_{-1}^1\int_0^1\frac{
 |x|+2}{(1+|x+y-sy|^2)^{\frac{1}{2}}(1+|x+y|^2)^{1/2}}
\, ds\, dy.
\end{align*}
If $|x|<4$, then clearly $H(x)\leq 12$. If $|x|>4$, we use the change of variables $y=|x|w$ ($dy = |x|dw$) to obtain that
\begin{align*}
H(x)&
\leq \int_{-|x|^{-1}}^{|x|^{-1}}\int_0^1
\frac{|x|+2}{|x|(|x|^{-2}+|\frac{x}{|x|}+w-sw|^2)^{\frac{1}{2}}
|x|(|x|^{-2}+|\frac{x}{|x|}+w|^2)^{1/2}}|x|
\, ds\, dw
\leq C|x|^{-1}(1+2|x|^{-1})\leq C|x|^{-1}.
\end{align*}

\end{proof}

Recall the definition
\begin{align*}
D^{log}(\rn)
:=\left\{
v\in L^2(\rn)\::\: 
\|v\|<\infty
\right\}.
\end{align*}
endowed with the norm $\|\cdot\|$ given by \eqref{normdef}. Note that $\|v\|_{L^2(\rn)}^2\leq \int_{\rn}v^2\ln(e+|x|^2)\, dx\leq \|v\|^2$.

\begin{lemma}
There is $C>0$ such that
\begin{align}\label{equiv2}
\|v\|\leq C\|v\|_{\D}\qquad \text{for every $v\in C_c^\infty(\rn)$}.
\end{align}
\end{lemma}
\begin{proof}
Let $u\in C_c^\infty(\S\backslash \{-e_{N+1}\})$ be such that $v=\iota(u)$. In coordinates given by $\left(\sigma\times\sigma,\S\backslash \{-e_{N+1}\}\times\S\backslash \{-e_{N+1}\}\right)$, we have that
\[
dV_{g\times g}(z,\zeta) = \phi(x)^N\phi(y)^N dx dy,\quad x=\sigma(z), \ y=\sigma(\zeta);
\]
hence, by \eqref{eq:4} and \eqref{eq:5},  
\begin{align}
\frac{2\cE(v,v)}{c_N}
&=\iint_{\{|x-y|\leq 1\}}\frac{|v(x)-v(y)|^2}{|x-y|^N}\, dy\, dx\notag\\
&=\iint_{\{|\sigma(z)-\sigma(\zeta)|\leq 1\}}\frac{|\phi^{\frac{N}{2}}(\sigma(z))u(z)-\phi^{\frac{N}{2}}(\sigma(\zeta))u(\zeta)|^2}{|z-\zeta|^N}\frac{(1+z_{N+1})^{N/2}(1+\zeta_{N+1})^{N/2}}{\phi(\sigma(z))^N\phi(\sigma(\zeta))^N}\, dV_g(\zeta)\, dV_g(z)\notag\\
&=\iint_{\{|\sigma(z)-\sigma(\zeta)|\leq 1\}}\frac{
\phi^{N}(\sigma(z))u(z)^2-2\phi^{\frac{N}{2}}(\sigma(z))\phi^{\frac{N}{2}}(\sigma(\zeta))u(\zeta)u(z)+\phi^{N}(\sigma(\zeta))u(\zeta)^2
}{|z-\zeta|^N\phi(\sigma(z))^{N/2}\phi(\sigma(\zeta))^{N/2}}\, dV_g(\zeta)\, dV_g(z)\notag\\
&=\iint_{\{|\sigma(z)-\sigma(\zeta)|\leq 1\}}\frac{
\frac{\phi^{N/2}(\sigma(z))}{\phi^{N/2}(\sigma(\zeta))}u(z)^2-2u(\zeta)u(z)+\frac{\phi^{N/2}(\sigma(\zeta))}{\phi^{N/2}(\sigma(z))}u(\zeta)^2
}{|z-\zeta|^N}\, dV_g(\zeta)\, dV_g(z)\notag\\
&=\iint_{\{|\sigma(z)-\sigma(\zeta)|\leq 1\}}\frac{
\left(\frac{\phi^{N/2}(\sigma(z))}{\phi^{N/2}(\sigma(\zeta))}-1\right)u(z)^2+(u(\zeta)-u(z))^2+\left(\frac{\phi^{N/2}(\sigma(\zeta))}{\phi^{N/2}(\sigma(z))}-1\right)u(\zeta)^2
}{|z-\zeta|^N}\, dV_g(\zeta)\, dV_g(z)\notag\\
&\leq \|u\|^2_{\H}+2\iint_{\{|\sigma(z)-\sigma(\zeta)|\leq 1\}} u(z)^2\frac{
\left|\frac{\phi^{N/2}(\sigma(z))}{\phi^{N/2}(\sigma(\zeta))}-1\right|}{|z-\zeta|^N}\, dV_g(\zeta)\, dV_g(z)\notag\\
&= \|u\|^2_{\H}+2\int_{\mathbb{S}^N}u(z)^2 \int_{U_1(z)}  \frac{
\left|\frac{\phi^{N/2}(\sigma(z))}{\phi^{N/2}(\sigma(\zeta))}-1\right|}{|z-\zeta|^N}\, dV_g(\zeta)\, dV_g(z) \leq C\|u\|^2_{\H}=C\|v\|_{\D}^2\label{bd1}
\end{align}
for some $C>0$, by Lemma \ref{h:lem} and  the isometric isomorphisms given in Theorem \ref{Ddefthm}. 

On the other hand, arguing as in Remark~\ref{l2:rmk}, observe that $\ln\left(e+|x|^2\right) \leq 3\ln\left(|x|+\frac{1}{|x|}\right)$ for $|x|\neq 0$.  Hence, we have, by \eqref{Pittc}, that 
\begin{align}
\int_{\rn}|v|^2 \ln (e+|x|^2) \, dx
&\leq
\frac{3}{2}\int_{\rn}|v|^2 \ln\left(|x| + \frac{1}{|x|}\right)^2 \, dx\leq 3\|v\|^2_{\D}.\label{bd2}
\end{align}
The claim now follows by \eqref{bd1} and \eqref{bd2}.
\end{proof}

The following result implies Theorem~\ref{Dlogprop}.

\begin{theorem}\label{Dlogprop2}
The space $D^{log}(\rn)$ is a Hilbert space with the inner product \eqref{sp}. Moreover, 
\begin{enumerate}
    \item $C_c^\infty(\mathbb{R}^N)$ is a dense subspace of $D^{log}(\rn)$,
    \item $D^{log}(\rn)\subset \D$ with continuous embedding,
    \item $D^{log}(\rn)$ is compactly embedded in $L^2(\rn)$.
    \item there is $C>0$ such that
    \begin{align}\label{dii}
        \iint_{|x-y|\geq 1} \frac{|v(y)||v(x)|}{|x - y|^N}\, dy\, dx\leq C\|v\|^2\qquad \text{for every $v\in D^{log}(\rn)$}.
    \end{align}
    \item there is $C>0$ such that, for every $v_1,v_2\in D^{log}(\rn)$, each term in \eqref{cELdef} is finite and
    \begin{align}\label{five}
        |\cE_L(v_1,v_2)|\leq C\|v_1\|\|v_2\|.
    \end{align}
\end{enumerate}
\end{theorem}
\begin{proof}
We first prove that \(D^{\log}(\mathbb R^N)\) is complete. Let
\((v_n)\subset D^{\log}(\mathbb R^N)\) be a Cauchy sequence with
respect to \(\|\cdot\|\). By the definition of the norm $\|\cdot\|$, \((v_n)\) is
a Cauchy sequence in $L^2\bigl(\mathbb R^N,\ln(e+|x|^2)\,dx\bigr)$ and therefore there is \(v\in L^2(\mathbb R^N,\ln(e+|x|^2)\,dx)\) such that
$v_n\longrightarrow v$ in $L^2\bigl(\mathbb R^N,\ln(e+|x|^2)\,dx\bigr).$ In particular, after passing to a subsequence, we may assume that
\(v_n(x)\to v(x)\) for almost every \(x\in\mathbb R^N\).  Let
\[
 d\mu(x,y)
 :=
 \frac{\chi_{\{(x,y)\in \R^N\times\R^N\::\:|x-y|\leq 1\}}}{|x-y|^N}\,dx\,dy\qquad \text{on \(\mathbb R^N\times\mathbb R^N\),}
\]
 and define
\(Tv(x,y):=v(x)-v(y)\). Since \((v_n)\) is a Cauchy sequence with respect to
\(\|\cdot\|\), the sequence \((Tv_n)\) is a Cauchy sequence in
\(L^2(\mathbb R^{2N},d\mu)\). Hence, for some
\(G\in L^2(\mathbb R^{2N},d\mu)\), we have $Tv_n\longrightarrow G$ in $L^2(\mathbb R^{2N},d\mu)$. Passing to a further subsequence, the convergence also holds
\(\mu\)-almost everywhere. On the other hand, by the almost everywhere
convergence of \(v_n\) to \(v\), we have that $ Tv_n(x,y)=v_n(x)-v_n(y)
 \longrightarrow v(x)-v(y)=Tv(x,y)$ for \(\mu\)-almost every \((x,y)\). Hence, \(G=Tv\)
\(\mu\)-almost everywhere. It follows that $\mathcal E(v_n-v,v_n-v)\longrightarrow 0.$ Therefore, $\|v_n-v\|\longrightarrow 0$ and \(v\in D^{\log}(\mathbb R^N)\), namely, \(D^{\log}(\mathbb R^N)\) is a Hilbert space.

The density of $C^\infty_c(\rn)$ in $D^{log}(\rn)$ follows by adapting the proof of the fractional case (see, for example, \cite[Theorem 3.1]{BGV21}).  We give the details in the appendix for completeness (see Proposition \ref{satz:dicht1}). 

Let $v\in C_c^\infty(\rn)$. By Pitt's inequality \eqref{Pittineq},
\begin{align*}
-\cE_L(v,v)
\leq \int_{\mathbb{R}^N}\ln(|x|^{2})|v(x)|^2 \,dx  -a_N\|v\|_{L^2(\rn)}^2,\qquad a_N:=2 \psi\left(\tfrac{N}{4}\right)+2\ln (2).
\end{align*}
Note also that, by \eqref{cELdef}, $\cE_L(|v|,|v|):=\cE(|v|,|v|)-c_N\iint_{\vert x - y\vert\geq 1}\frac{|v(x)||v(y)|}{\vert x-y\vert^N}\; dxdy + \rho_N\int_{\mathbb{R}^N}|v|^2\ dx.$ Hence,
\begin{align}
b(v,v)
&:=c_N \iint_{\{|x-y|\geq 1\}} \frac{|v(y)||v(x)|}{|x - y|^N}\notag\\
&=\cE(|v|,|v|)-\cE_L(|v|,|v|) + \rho_N\int_{\mathbb{R}^N}|v|^2\ dx\notag\\
&\leq 
\cE(|v|,|v|)
+\left(\int_{\mathbb{R}^N}\ln(|x|^{2})|v(x)|^2
 -a_N\|v\|_{L^2(\rn)}^2\right)
+ \rho_N\int_{\mathbb{R}^N}|v|^2\ dx \notag\\
&\leq 
\cE(|v|,|v|)
+ d_N\|v\|_{L^2(\rn)}^2
+\int_{\mathbb{R}^N}\ln(|x|^{2})|v(x)|^2  \notag\\
&\leq \cE(v,v)
+d_N \|v\|_{L^2(\rn)}^2
+\int_{\{|x|>1\}}\ln(|x|^{2})|v(x)|^2\notag\\
&\leq \cE(v,v)
+(d_N+1)\int_{\rn}\ln(e+|x|^{2})|v(x)|^2
\leq (d_N+1)\|v\|^2
\label{bound1}
\end{align}
with $d_N:=|\rho_N-a_N|
=|\psi(\frac{N}{2})+\Gamma'(1)-2 \psi\left(\tfrac{N}{4}\right)|$, and where we used that 
\begin{align*}
\cE(|v|,|v|)
&=\frac{c_N}{2}\iint_{\{|x-y|<1\}} \frac{||v(x)| - |v(y)||^2}{|x - y|^N}
=\frac{c_N}{2}\iint_{\{|x-y|<1\}} \frac{v(x)^2 -2|v(x)v(y)|+v(y)^2}{|x - y|^N} \\
&\leq \frac{c_N}{2}\iint_{\{|x-y|<1\}} \frac{v(x)^2 -2v(x)v(y)+v(y)^2}{|x - y|^N}
=\cE(v,v).
\end{align*}
Therefore, there is $C>0$ such that \eqref{dii} holds for every  $v\in C_c^\infty(\rn)$, and the statement for $v\in D^{log}(\rn)$ follows by density and by Fatou's Lemma. 

Now, by \eqref{cELdef} and \eqref{bound1}, for every $v\in C^\infty_c(\rn)$,
\begin{align*}
|\cE_L(v,v)|
\leq \cE(v,v)+b(v,v)+|\rho_N|\|v\|_{L^2(\rn)}^2
\leq C_1\|v\|^2
\end{align*}
for some constant $C_1=C_1(N)>0$.  Moreover, 
\begin{align*}
\left|\int_{\rn}|v|^2 \ln \phi^{-2} \, dx\right|
=\left|2\int_{\rn}|v|^2 \ln \left(\frac{1+|x|^2}{2}\right) \, dx\right|
\leq 2\int_{\rn}|v|^2 \ln(1+|x|^2) \, dx+2\ln 2\|v\|_{L^2(\rn)}^2
\leq C_2\|v\|^2
\end{align*}
for some constant $C_2=C_2(N)>0$. Hence, 
\begin{align}\label{dineq}
\|v\|_{\D}\leq C\|v\|
\end{align}
 for some $C=C(N,\kappa)>0$, namely, $D^{log}(\rn)\subset \D$ with continuous embedding. This, together with Theorem~\ref{Ddefthm}, implies the compact embedding of $D^{log}(\rn)$ into $L^2(\rn)$.

Finally, to show \eqref{five}, let $v_1,v_2\in C^{\infty}_c(\rn),$ by Cauchy-Schwarz inequality and \eqref{dineq},
\begin{align*}
|\cE_L(v_1,v_2)|-\left|\int_{\rn}v_1 v_2(\kappa-A_N+\ln \phi^{-2})\right|\leq 
|\langle v_1,v_2\rangle_{\D}|\leq \|v_1\|_{\D}\|v_2\|_{\D}\leq \|v_1\|\|v_2\|.
\end{align*}
Using the definition of $\|\cdot\|$, $\left|\int_{\rn}v_1 v_2(\kappa-A_N+\ln \phi^{-2})\right|
\leq C\|v_1\|\|v_2\|$ for some constant $C$. Hence, $|\cE_L(v_1,v_2)|\leq C\|v_1\|\|v_2\|$ and the claim now follows by density. 
\end{proof}

\begin{proof}[Proof of Theorem~\ref{thmD}]
The claim \eqref{equiv} follows from \eqref{equiv2} and \eqref{dineq}. Hence, by Lemma~\ref{density:lem2} and Theorem~\ref{Dlogprop2},
\begin{align*}
D^{log}(\rn)= \overline{C^\infty_c(\rn)}^{\|\cdot\|}
=\overline{\iota (C_c^\infty(\S\backslash\{-e_{N+1}\}))}^{\|\cdot\|}
\simeq\overline{\iota (C_c^\infty(\S\backslash\{-e_{N+1}\}))}^{\|\cdot\|_{\D}}=\D,
\end{align*}
where $\simeq$ means ``is isomorphic to".
\end{proof}

\section{Logarithmic Yamabe-type problems}\label{Yamabe:sec}

We are ready to show Theorem~\ref{equivalentProblems}.
\begin{proof}[Proof of Theorem~\ref{equivalentProblems}]
Let $u$ be a weak solution of \eqref{Y1} and let $\vartheta\in \mathbb{D}(\mathbb{R}^N)\cap C_c^\infty(\rn)$. 
Furthermore, let $\varphi\in \mathbb{H}(\mathbb{S}^N)\cap C^\infty(\S)$ be such that 
$\vartheta=\iota(\varphi)$. It is easy to check that
\begin{align*}
    \vartheta\in L^1_0(\rn):=\left\{ w\in L^1_{loc}(\rn)\::\: \int_{\rn}\frac{|w(x)|}{(1+|x|)^N}\, dx<\infty  \right\}.
\end{align*}
Hence, by \cite[Proposition 1.3]{CW19} we have that $L_\Delta \vartheta$ is well-defined in $\rn.$  By Proposition~\ref{Prop:ConformalWithEuclidean},
\[
\mathscr{P}_g^{\log} \varphi \circ\sigma^{-1} = \phi^{-\frac{N}{2}}L_\Delta \vartheta - 2 \phi^{-\frac{N}{2}}\vartheta \ln\phi\quad \text{ in }\rn.
\]
Since $u$ is a weak solution of \eqref{Y1}, $\int_{\S}(\tfrac{4}{N}u\ln|u| + \mu u)\varphi
= \int_{\S} u \mathscr{P}_g^{\log} \varphi.$ Now, using \eqref{volumeform},
\begin{align*}
    \int_{\S} u \mathscr{P}_g^{\log} \varphi
    = \int_{\rn} \phi^{-\frac{N}{2}}v \left(\phi^{-\frac{N}{2}}L_\Delta \vartheta - 2 \phi^{-\frac{N}{2}}\vartheta \ln\phi\right)\phi^{N}
    = \int_{\rn} v \left(L_\Delta \vartheta+\vartheta \ln\phi^{-2}\right)
\end{align*}
and, similarly,
\begin{align*}
\int_{\S}(\tfrac{4}{N}u\ln|u| + \mu u)\varphi 
=   \int_{\rn}\left(\tfrac{4}{N} v \ln|\phi^{-\frac{N}{2}}v| + \mu v\right)\vartheta
=   \int_{\rn}\left(v (\ln\phi^{-2}+\tfrac{4}{N}\ln|v|) + \mu v\right)\vartheta.
\end{align*}
Hence,
\begin{align*}
    \int_{\rn} v L_\Delta \vartheta+\int_{\rn}v\vartheta \ln\phi^{-2}
    =\int_{\rn}v\vartheta  (\ln\phi^{-2}+\tfrac{4}{N}\ln|v|) + \mu v \vartheta,
\end{align*}
namely, $\int_{\rn} v L_\Delta \vartheta
    =\int_{\rn}\tfrac{4}{N}v\vartheta\ln|v|+ \mu v \vartheta$ 
and $v$ is a weak solution of \eqref{Y2}.

On the other hand, if $v\in \D$ is a weak solution of \eqref{Y2} and $\varphi\in \mathbb{H}(\mathbb{S}^N)\cap C_c^\infty(\S\backslash\{-e_{N+1}\})$, let $\vartheta\in \mathbb{D}(\mathbb{R}^N)\cap C_c^\infty(\rn)$ satisfy $\vartheta=\iota(\varphi)$. Now we can argue as before to yield that $u$ is a weak solution of \eqref{Y1}.
\end{proof}

\begin{remark}\label{change}
\emph{
A direct calculation shows that a function $v:\mathbb{R}^N\rightarrow\mathbb{R}$ is a solution of 
\begin{align*}
L_\Delta v = \frac{4}{N} v \ln|v| + \mu v\quad \text{in}\ \mathbb{R}^N
\end{align*}
if and only if $w:=Kv$ is a solution of
\begin{equation}\label{Problem:ChenZhouEquation}
L_\Delta w = \frac{4}{N}w\ln |w|\quad \text{in}\ \mathbb{R}^N
\end{equation}
with $K=e^{\frac{N}{4}\mu}$.    
}
\end{remark}

\begin{proof}[Proof of Theorem~\ref{classthm}]
Let $v$ be a solution of \eqref{Y2} for some $\mu\in \R$. By Remark \ref{change}, $\widetilde v:= e^{\frac{N}{4}(\mu-A_N)}v$ is a solution of \eqref{Y2} with $\mu=A_N$ and note that \eqref{frankprob} coincides with \eqref{Y1} for $\mu=A_N$.  By Theorem~\ref{equivalentProblems}, we have that $
u=\iota^{-1}(\widetilde v)$ is a weak solution of \eqref{Y1}. By \cite[Theorem 1]{FKT20}, $u=u_\theta$ is of the form \eqref{franksol} for some $\theta\in \mathbb R^{N+1}$. But, inspecting the proof of \cite[Theorem 1]{FKT20} more can be said. The vector $\theta$ can be written as $\theta=\theta(a,b)= \frac{2\eta - b^2(1+\eta_{N+1})e_{N+1}}{2+b^2(1+\eta_{N+1})},$ where $\eta:=\sigma^{-1}(a)
 = \left(\frac{2a}{1+\vert a\vert^2}, \frac{1-\vert a\vert^2}{1+\vert a\vert^2}\right)$ for some $b>0$ and $a\in \rn$. Therefore, $\widetilde v(x)=\iota(u_\theta)=v_{a, t}(x)
=\left(\frac{2t}{t^2+|x-a|^2}\right)^\frac{N}{2}$ (see \cite[page 6]{FKT20}) and 
$v(x)=e^{\frac{N}{4}(A_N-\mu)}\left(\frac{2t}{t^2+|x-a|^2}\right)^\frac{N}{2}$ for $x\in \rn$.
\end{proof}

\begin{remark}\label{constantsolsrmk}
\emph{
In particular, for $\mu=A_N=2 \psi\left(\tfrac{N}{2}\right)$, the function $u\equiv 1$ is a solution of \eqref{Y1} and 
\[
v(x)= \phi(x)^\frac{N}{2}=\left(\frac{2}{1+\vert x\vert^2}\right)^{\frac{N}{2}}
\]
is a solution to the problem \eqref{Y2}.  Curiously, Remark~\ref{explicit:rmk} also shows that $v$ is a solution of the linear eigenvalue problem \eqref{E2}. For general $\mu\neq A_N$, the constant $u\equiv e^{\frac{N}{4}(A_N-\mu)}$ is a solution of \eqref{Y1} and 
\[
v(x)=e^{\frac{N}{4}(A_N-\mu)}\phi^{\frac{N}{2}}(x) = e^{\frac{N}{4}(A_N-\mu)}\left(\frac{2}{1+\vert x\vert^2}\right)^{\frac{N}{2}}
\]
is a solution of \eqref{Y2}.
}
\end{remark}

\begin{remark}\label{chen:rmk}
\emph{
In \cite{CZ24}, nonnegative classical solutions of 
\begin{align}\label{chenprob}
L_\Delta w = \frac{4}{N} w \ln w\qquad \text{ in }\mathbb R^N.
\end{align}
were considered. As mentioned in the introduction, here a classical solution is a Dini continuous function in $\rn$ such that 
$\int_{\rn}\frac{|u(x)|}{(1+|x|)^N}\, dx<\infty$ and satisfies \eqref{chenprob} pointwisely. In \cite[Theorem 1.1]{CZ24} it is shown that all such solutions are given by the following family of functions
\begin{align}\label{chenf}
w_{\widetilde x, t}(x) :=\gamma_N\left(\frac{2t}{t^2+|x-\widetilde x|^2}\right)^\frac{N}{2},\qquad \gamma_N:=e^{\frac{N}{2}\psi(\frac{N}{2})},\qquad \widetilde x\in \rn.
\end{align}
Let us explain how the constant $\gamma_N$ fits in this context. Let us show first that the functions $w_{\widetilde x, t}$ do belong to the space $\mathbb{D}(\mathbb{R}^N)$. For this, we use the following bijection (see \cite[page 6]{FKT20}). For $b>0$ and $a\in \rn$, let
 \begin{align*}
 \theta(a,b) := \frac{2\eta - b^2(1+\eta_{N+1})e_{N+1}}{2+b^2(1+\eta_{N+1})},\qquad \eta:=\sigma^{-1}(a)
 = \left(\frac{2a}{1+\vert a\vert^2}, \frac{1-\vert a\vert^2}{1+\vert a\vert^2}\right),
 \end{align*}
 then
\begin{align*}
w_{\widetilde x, t}(x)
= \gamma_N \iota(u_{\theta(\widetilde x,t)})(x)
= \gamma_N \phi(x)^\frac{N}{2}u_{\theta(\widetilde x,t)}\circ \sigma^{-1}(x)
\qquad \text{ for all }x\in \mathbb R^N.
\end{align*}
Since $u_{\theta(\widetilde x,t)}\in C^\infty(\S)$, we have that $w_{\widetilde x, t}\in \mathbb{D}(\mathbb{R}^N)$.  Now, applying Theorem~\ref{equivalentProblems} (with $\mu=A_N= 2 \psi\left(\tfrac{N}{2}\right)$), we have that $v:=\iota(u_{\theta(\widetilde x,t)})$ is a solution of 
\begin{align*}
L_\Delta v 
= \frac{4}{N} v \ln(v) + A_N v 
= \frac{4}{N} v \ln(e^{\frac{N}{4}A_N} v)
= \frac{4}{N} v \ln(\gamma_N v)
\qquad \text{ in }\rn.
\end{align*}
Multiplying this equation by $\gamma_N$ we see that $\gamma_N v = w_{\widetilde x, t}$ is a solution of \eqref{Y2} with $\mu=0$.   
}
\end{remark}

\begin{remark}
\emph{
Beckner’s logarithmic Sobolev inequality on $\S$ with sharp constant \cite{B92,B97} states that
\begin{equation}
\iint_{\mathbb S^N \times \mathbb S^N} \frac{|v(z) - v(\zeta)|^2}{|z - \zeta|^N} \, dV_g(\zeta)\, dV_g(z) 
\geq C_N \int_{\mathbb S^N} |v(z)|^2 \ln \frac{|v(z)|^2 |\mathbb S^N|}{\|v\|_2^2} \, dV_g(z)
\label{eq:beckner}
\end{equation}
with $C_N= \frac{4}{N}\frac{\pi^{N/2}}{\Gamma(N/2)}$. In \cite{B97} Beckner showed that equality holds in \eqref{eq:beckner} if and only if
\begin{equation}
v(z) = c \left( \frac{\sqrt{1 - |\omega|^2}}{1 - z\cdot \omega} \right)^{N/2}
\label{eq:equality}
\end{equation}
for some $\omega \in \mathbb{R}^{N+1}$ with $|\omega| < 1$ and some $c \in \mathbb{R}$.  Equation \eqref{frankprob} appears as a suitable normalization of the Euler--Lagrange equation of this optimization problem.  Note that the function $u\equiv \gamma_N$ is an optimizer of Beckner’s logarithmic Sobolev inequality on $\S$. 
}
\end{remark}

\section{The conformal logarithmic Laplacian on the cylinder}\label{app:cyl}

In this section we show how the same techniques from Theorem \ref{main:thm} can be used to obtain an explicit formula for the conformal logarithmic Laplacian on a cylinder. Let $N\geq2$, $\mathscr{C}:=\mathbb{R}\times\mathbb{S}^{N-1}\approx\mathbb{R}^N\smallsetminus\{0\}$ be the cylinder with polar coordinates $x=(r,\theta)\in\mathbb{R}\times\mathbb{S}^{N-1}$, where $r=e^t$ for some $t\in\mathbb{R}$. We endow $\mathscr{C}$ with the metric $h_0=dt^2+g$, where $g$ is the Riemannian metric on the sphere $\mathbb{S}^{N-1}$. Hence, the Riemannian volume form is simply $dV_{h_0}= dt\;dV_g$. Following \cite[Section 2]{AoDelaTorreGonzalez2022}, the conformal fractional Laplacian of order $s\in(0,1)$ on the cylinder $\mathscr{C}$ is 
\begin{equation}\label{FractionalConformalLaplacianCylinder}
\begin{split}
\mathscr{P}_{h_0}^s u(t,\theta)&:= 
c_{N,s}\text{P.V.}\int_{\mathscr{C}}K_s(t,\tau,\theta,\vartheta)(u(t,\theta)-u(\tau,\vartheta)) \; dV_{h_0}(\tau,\vartheta) + B_{N,s}u(t,\theta)\\
&=c_{N,s}\text{P.V.}\int_\mathbb{R}\int_{\mathbb{S}^{N-1}}K_s(t,\tau,\theta,\vartheta)(u(t,\theta)-u(\tau,\vartheta)) \; dV_{g}(\vartheta) \;d\tau + B_{N,s}u(t,\theta),
\end{split}
\end{equation}
where $u\in\mathcal{C}_c^\infty(\mathscr{C})$, $c_{N,s}$ is the same constant given in \eqref{Definition:COnstantAs}, 
\begin{equation}\label{DefinitionB_N,s}
B_{N,s}:= 4^s\left( \frac{\Gamma(\frac{N}{4} + \frac{s}{2})}{\Gamma(\frac{N}{4} - \frac{s}{2})} \right)^2>0,
\end{equation}
and the kernel $K_s$ is given by
\begin{equation}\label{KernelCylinder}
K_s(t,\tau,\theta,\vartheta):= \frac{ e^{-\frac{N+2s}{2}\vert t - \tau\vert} }{\left[1+e^{-2\vert t-\tau\vert}-2 e^{-\vert t-\tau \vert}\theta\cdot\vartheta\right]^{\frac{N+2s}{2}}} = \frac{1}{\left[ 2\cosh\vert t-\tau\vert - 2\theta\cdot\vartheta \right]^{\frac{N+2s}{2}}},
\end{equation}
where $\theta\cdot\vartheta$ refers to the usual inner product in $\mathbb{R}^N$. 

The conformal covariance property of this operator allows to write the conformal fractional Laplacian on the cylinder via the fractional Laplacian on $\mathbb{R}^N$, see \cite[Section 6.5]{DM18} of \cite[Section 2]{AoDelaTorreGonzalez2022}; indeed, if $u\in \mathcal{C}_c^\infty(\mathscr{C})$, we have the intertwining identity
\begin{equation}
\mathscr{P}^s_{h_0}u = r^{\frac{N+2s}{2}}(-\Delta)^s(r^{-\frac{N-2s}{2}}u).
\end{equation}

We have the following result.

\begin{theorem}\label{cyl:thm}
Let $N\geq 2$ and $u\in \mathcal{C}_c^\infty(\mathscr{C})$. The conformal logarithmic Laplacian on $(\mathscr{C},h_0)$ is given by
\begin{equation*}
\mathscr{P}_{h_0}^{\log}u(t,\theta):= \left.\frac{d}{ds}\right|_{s=0} [\mathscr{P}^s_{h_0}u](t,\theta)
=c_N \int_\mathbb{R}\int_{\mathbb{S}^{N-1}}\frac{u(t,\theta) - u(\tau,\vartheta)}{\vert 2\cosh\vert t- \tau \vert - 2\theta\cdot\vartheta \vert^{N/2}} dV_g(\vartheta)\;d\tau + B_N \, u(t,\theta),
\end{equation*}
where
\begin{equation*}
B_N:= \left.\frac{d}{ds}\right|_{s=0} B_{N,s} = \log 4 + 2\psi\left(\frac{N}{4}\right),
\end{equation*}
and satisfies the following intertwining identity:
\begin{equation}
\mathscr{P}_{h_0}^{\log}u = r^{N/2}L_{\Delta}(r^{-\frac{N}{2}}u) + 2u\log r,\quad r:= e^{t}.
\end{equation}
Moreover, the associated logarithmic $Q$-curvature is given by
\[
Q_{h_0}^{\log} = \mathscr{P}_{h_0}^{\log}(1)= B_N.
\]
\end{theorem}
\begin{proof}
The integral formula follows by arguing as in the proof of Theorem~\ref{main:thm} using \eqref{FractionalConformalLaplacianCylinder} and \eqref{KernelCylinder}. The intertwining identity follows as in the proof of Proposition~\ref{Prop:ConformalWithEuclidean}. Finally, since $\mathscr P_{h_0}^s(1)=B_{N,s}$, we have that $Q_{h_0}^{\log}=B_N$.
\end{proof}

The formula in Theorem \ref{cyl:thm} simplifies greatly if the function is independent of $\theta$. The following is a radial integration formula for the conformal logarithmic laplacian on the cylinder.

\begin{proposition} If $u\in\mathcal{C}_c^\infty(\mathscr{C})$ is radial, i.e., $u(t,\theta)=w(t)$ for some $w\in\mathcal{C}^\infty_c(\mathbb{R})$, then
\begin{equation}\label{One-DimensionalLaplacian}
\mathscr{P}_{h_0}^{\log}u(t,\theta) = \int_\mathbb{R} \left( w(t)-w(\tau)\right) \frac{e^{-\frac{N-2}{2}\vert t-\tau \vert}}{\sinh\vert t-\tau\vert}  \; d\tau + B_N w(t).
\end{equation}
\end{proposition}

\begin{proof}
Denote by $P$ the Poisson kernel on the unit ball $B_1^N(0)$ in $\mathbb{R}^N$, i.e., if $x=\rho\theta$, with $\rho\geq 0$ and $\theta\in\mathbb{S}^{N-1}$, then
\[
P(x,\vartheta) = \frac{1-\vert x\vert^2}{\vert \mathbb{S}^{N-1}\vert\vert x - \vartheta\vert^N} = \frac{1-\rho^2}{\vert \mathbb{S}^{N-1}\vert\vert \rho\theta - \vartheta\vert^N},\qquad x=\rho\theta\in B_1^N(0), \quad \vartheta\in\mathbb{S}^{N-1},  
\]
where $\vert \mathbb{S}^{N-1}\vert$ denotes the surface area of the sphere $\mathbb{S}^{N-1}.$

Then, using the representation formula for the function $v(x)\equiv 1$, which is a solution to the problem
\[
\Delta v = 0 \quad\text{in} \quad B_1^N(0),\qquad v=1 \quad \text{ on } \partial B_1^N(0)=\mathbb{S}^{N-1},
\]
we have that
\[
\int_{\mathbb{S}^{N-1}}P(\rho\theta,\vartheta) dV_g(\vartheta)= 1,\qquad \text{for all }x=\rho\theta\in B_1^N(0).
\]

With this, for any $t,\tau\in\mathbb{R}$, taking $\rho:=e^{-\vert t-\tau\vert}\in[0,1)$, we have that
\begin{align*}
\int_{\mathbb{S}^{N-1}}\frac{1}{\left[ 2\cosh\vert t-\tau\vert -2 \theta\cdot\vartheta\right]^{N/2}} \;dV_g(\vartheta) 
&= \int_{\mathbb{S}^{N-1}}\frac{1}{\left[  e^{-(t+\tau)} \vert e^t\theta - e^{\tau\vartheta}\vert^2  \right]^{N/2}} \;dV_g(\vartheta)
= \int_{\mathbb{S}^{N-1}}\frac{1}{\left[  \rho^{-1} \vert \rho\theta - \vartheta\vert^2  \right]^{N/2}} \;dV_g(\vartheta)\\
&= \frac{\vert \mathbb{S}^{N-1}\vert\rho^{N/2}}{1-\rho^2} \int_{\mathbb{S}^{N-1}} P(\rho\theta,\vartheta)\;dV_g(\vartheta)
= \frac{\vert \mathbb{S}^{N-1}\vert\rho^{N/2}}{1-\rho^2} \\
&= \vert \mathbb{S}^{N-1}\vert \frac{ e^{-\frac{N}{2}\vert t-\tau\vert}}{1-e^{-2\vert t - \tau\vert}} 
= \vert \mathbb{S}^{N-1}\vert \frac{e^{-\frac{N-2}{2}\vert t-\tau\vert}}{2\sinh\vert t -\tau\vert}.
\end{align*}
Also notice that $\vert \mathbb{S}^{N-1}\vert c_N = \frac{N\pi^{N/2}}{\Gamma(\frac{N}{2}+1)}c_{N}= 2.$ Hence, if $u(t,\theta)=w(t),$ we get that
\begin{align*}
\mathscr{P}_{h_0}^{\log}u(t,\theta)& = c_N \int_\mathbb{R}\int_{\mathbb{S}^{N-1}}\frac{w(t) - w(\tau)}{\vert 2\cosh\vert t- \tau \vert - 2\theta\cdot\vartheta \vert^{N/2}} dV_g(\vartheta)\;d\tau + B_N u(t,\theta)\\
&= c_N  \vert \mathbb{S}^{N-1}\vert \int_\mathbb{R}(w(t) - w(\tau))\frac{e^{-\frac{N-2}{2}\vert t-\tau\vert}}{2\sinh\vert t -\tau\vert}\;d\tau + B_N u(t,\theta)
= \int_\mathbb{R} \left( w(t)-w(\tau)\right) \frac{e^{-\frac{N-2}{2}\vert t-\tau \vert}}{\sinh\vert t-\tau\vert}  \; d\tau + B_N w(t),
\end{align*}
as claimed.
\end{proof}

\appendix

\section{Density results}

In this appendix we show the density of $C^\infty_c(\rn)$ in $D^{log}(\rn)$, by adapting some  arguments of the fractional case and using the theory of weighted Lebesgue spaces. 
\begin{lemma}\label{lemma:dicht-vorarbeit}
	Let $v\in D^{log}(\rn)$. For every $\eps>0$  there is $\vartheta\in C^{0,1}_c(\R^N)$ such that $\|v-\vartheta v\|<\varepsilon$.
\end{lemma}
\begin{proof}
Let $v\in D^{log}(\rn)\subset L^2(\rn)$. For $n\in \N$ let $\vartheta_n\in C^{0,1}_c(\R^N)$ be radially symmetric and such that $\vartheta_n\equiv 1$ on $B_n(0)$, $\vartheta_n\equiv 0$ on $B_{n+1}(0)^c$. Clearly, we may assume that $[\vartheta_n]_{C^{0,1}(\R^N)}=1$. Arguing as in Lemma~\ref{Lemma:LogNormConstants}, it is not hard to show that there is $C>0$ such that $\|\vartheta_n v\|\leq C\|v\|$ for all $n\in \N$. In the following, let $v_n:=\vartheta_nv$ and we assume $v\geq 0$ without loss of generality. Then $0\leq v-v_n\leq v$ on $\R^N$ and $v-v_n=0$ on $B_n(0)$. 
Note that
\begin{align*}	
|v(x)(1-\vartheta_n(x))-v(y)(1-\vartheta_n(y))|^2
&\leq (|v(x)-v(y)|(1-\vartheta_n(x)) +|v(y)||\vartheta_n(x)-\vartheta_n(y)|)^2\\
&\leq 2|v(x)-v(y)|^2 +2|v(y)|^2\min\{1,|x-y|^2\}=:U(x,y).
\end{align*}
Since
\begin{align*}
\iint_{|x-y|<1} \frac{U(x,y)}{|x-y|^N}\, dy
\leq \frac{4}{c_N}\cE(v,v)+2\int_{\rn} |v(y)|^2 \int_{B_1(y)}|x-y|^{2-N}\, dx\, dy
=\frac{4}{c_N}\cE(v,v)+|\S|\int_{\rn} |v(y)|^2 \, dy<\infty,
\end{align*}
we have, by dominated convergence, that $\cE(v-v_n,v-v_n)\to 0$ as $n\to\infty$. Moreover,
\begin{align*}
\int_{\rn}(v-v_n)^2\ln(e+|x|^2)\, dx\leq \int_{\rn}v^2\ln(e+|x|^2)\, dx<\infty,
\end{align*}
using dominated convergence again we have that $\int_{\rn}(v-v_n)^2\ln(e+|x|^2)\, dx\to 0$ as $n\to\infty$ and thus $\|v-v_n\|\to 0$ for $n\to \infty,$ as claimed.
\end{proof}

A locally integrable function $w:\mathbb{R}^N\to(0,\infty)$ belongs to the Muckenhoupt class $A_2(\mathbb{R}^N)$ if
\begin{align}\label{A2est}
[w]_{A_2}
:=
\sup_{B\subset\mathbb{R}^N}
\left(\frac{1}{|B|}\int_B w(x)\,dx\right)
\left(\frac{1}{|B|}\int_B \frac{1}{w(x)}\,dx\right)
<\infty,    
\end{align}
where the supremum is taken over all Euclidean balls
$B\subset\mathbb{R}^N$, see \cite{T00}. We now show that the logarithmic weight on the norm $\eqref{normdef}$ belongs to the $A_2$ Muckenhoupt class.

\begin{lemma}\label{A2}
There is $C>0$ such that 
\begin{align*}
H(r):=\left(\frac{1}{r^{N}}\int_0^r t^{N-1}\ln(e+t^2)\, dt \right)\left(\frac{1}{r^{N}}\int_0^r t^{N-1}\frac{1}{\ln(e+t^2)}\, dt \right)\leq C\qquad \text{ for all }r>0.
    \end{align*}
In particular, the weight $w(x):=\ln(e+|x|^2)$ belongs to the $A_2$ Muckenhoupt class.
\end{lemma}
\begin{proof}
Note that $H(r)$ is bounded for $r\in (0,1)$, because $\ln(e+r^2)\in(1,\ln(e+1))$. By monotonicity, 
\begin{align*}
\frac{1}{r^{N}}\int_0^r t^{N-1}\ln(e+t^2)\, dt\leq N^{-1}\ln(e+r^2)
\quad \text{ and }\quad \frac{1}{r^{N}}\int_0^r \frac{t^{N-1}}{\ln(e+t^2)}\, dt\leq \frac{1}{r}\int_0^r \frac{1}{\ln(e+t^2)}\, dt.   
\end{align*}
By L'Hôpital's rule, 
\begin{align*}
N\lim_{r\to \infty}H(r)\leq
\lim_{r\to \infty}\frac{\ln(e+r^2)}{r}\int_0^r \frac{1}{\ln(e+t^2)}\, dt
=\lim_{r\to \infty}\frac{\frac{d}{dr}\int_0^r \frac{1}{\ln(e+t^2)}\, dt}{\frac{d}{dr}\frac{r}{\ln(e+r^2)}}
=\lim_{r\to\infty}\frac{\frac{1}{\ln(e+r^2)}}{\frac{1}{\ln\left(e+r^2\right)}-\frac{2 r^2}{\left(e+r^2\right) (\ln\left(r^2+e\right))^2}}=1.
\end{align*}
Hence, $H(r)$ is bounded as $r\to\infty$ and \eqref{A2est} follows for balls centered at the origin. This shows the bound \eqref{A2est} for balls $B$ centered at the origin. The argument for arbitrary balls can be argued similarly. Hence, $w\in A_2(\mathbb{R}^N)$.
\end{proof}

\begin{lemma}\label{mollifier:lemma}
	Let $v\in D^{log}(\rn)$ and let $(\rho_\eps)_{\eps\in(0,1]}$ be a Dirac sequence. Then $v_{\eps}:=\rho_\eps\ast v\to v $ in $D^{log}(\rn)$ as $\eps\to 0$.
\end{lemma}
\begin{proof}
 	By Jensen's inequality,
	\begin{align}
\cE(v_{\eps},v_{\eps})
&=\frac{c_N}{2}\int_{\R^N}\int_{B_1(x)} \Big| \int_{\R^N} \rho_{\eps}(z)v(x-z)\ dz-\int_{\R^N}\rho_{\eps}(z)v(y-z)\ dz\Big|^2 \frac{dy\,dx}{|x-y|^{N}}\notag\\
	&=\frac{c_N}{2}\int_{\R^N}\int_{B_1(x)} \Big| \int_{\R^N} \rho_{\eps}(z)[v(x-z)-v(y-z)]\ dz\Big|^2 \frac{dy\,dx}{|x-y|^{N}}\notag\\
	&\leq \frac{c_N}{2}\int_{\R^N}\int_{B_1(x)}\,\int_{\R^N} \rho_{\eps}(z)|v(x-z)-v(y-z)|^2\ dz \frac{dy\, dx}{|x-y|^{N}}\notag\\
	&=\frac{c_N}{2}\int_{\R^N}\rho_{\eps}(z)\int_{\R^N}\int_{B_1(x)}  |v(x-z)-v(y-z)|^2\frac{dy\, dx}{|x-y|^{N}}\ dz=
    \|\rho_{\eps}\|_{L^1(\rn)}\cE(v,v)=\cE(v,v).\label{upbd}
	\end{align}
    Note that $\cE(v-v_{\eps},v-v_{\eps})= \frac{c_N}{2}\|V-V_\eps\|^2_{L^2(\rn\times B_1(0))}$ with 
	\[
	V(x,y):=\frac{v(x)-v(x+y)}{|y|^{\frac{N}{2}}}\quad\text{and}\quad V_{\eps}(x,y):=\frac{v_{\eps}(x)-v_{\eps}(x+y)}{|y|^{\frac{N}{2}}}.
	\]

Let $Q:=\R^N\times B_1(0)$. By Fatou's Lemma and \eqref{upbd},
	\begin{align}\label{l2}
\|V\|_{L^2(Q)}^2\leq \liminf_{\eps\to 0}\|V_{\eps}\|_{L^2(Q)}^2
    \leq \limsup_{\eps\to 0}\|V_{\eps}\|_{L^2(Q)}^2
    =\frac{2}{c_N}\limsup_{\eps\to 0}\cE(v_{\eps},v_{\eps})
    \leq \frac{2}{c_N}\cE(v,v)=\|V\|_{L^2(Q)}^2.
	\end{align}
   
	Let $U\in L^{2}(Q)$ and define
	\[
	U_n(x,y):=\left\{\begin{aligned} &U(x,y)&& \text{for $|x|<n$ and $|y|\geq \frac{1}{n}$;}\\
	&0 &&\text{otherwise.}
	\end{aligned}\right.
	\]
	Then $U_n\to U$ in $L^{2}(Q)$ for $n\to \infty$ by dominated convergence. Next, let
	\[
	k_1^n,k_2^n:\R^N\to \R,\quad k_1^n(x)=\int_{B_1(0)}|y|^{-\frac{N}{2}}U_n(x,y)\ dy,\quad 
        k_2^n(y)=\int_{\R^N}|x|^{-\frac{N}{2}}U_n(y-x,x)\ dx.
	\]
	Then, by Fubini's Theorem,
	\begin{align*}
	\lim_{\eps\to 0}&\iint_{Q} V_{\eps}(x,y)U_n(x,y)\, dy\, dx
    =\lim_{\eps\to 0}\iint_{Q}(v_{\eps}(x)-v_{\eps}(x+y)) \frac{U_n(x,y)}{|y|^{\frac{N}{2}}}\, dy\, dx\\
	&= \lim_{\eps\to 0}\int_{\R^N}v_{\eps}(x)k_1^n(x)\ dx- \lim_{\eps\to 0}\int_{\rn}v_{\eps}(y)k_2^n(y)\ dy\\
    &=\int_{\R^N}v (x)k_1^n(x)\ dx-  \int_{\R^N}v (y)k_2^n(y)\ dy=\iint_{Q} V(x,y)U_n(x,y)\ dxdy.
	\end{align*}
	Thus,
	\begin{align*}
	\limsup_{\eps\to 0}&\left|\ \ \iint_{Q} (V_{\eps}(x,y)-V(x,y))U(x,y)\ dxdy\right|\\
	&=\limsup_{\eps\to 0}\Bigg|\ \,\iint_{Q} (V_{\eps}(x,y)-V(x,y))(U(x,y)-U_n(x,y))\ dxdy\Bigg|\\
	&\leq \limsup_{\eps\to 0}\|V_{\eps}-V\|_{L^2(Q)}\|U-U_n\|_{L^{2}(Q)}\leq 2\|V\|_{L^2(\R^N)}\|U-U_n\|_{L^{2}(Q)}.
	\end{align*}
	Therefore, 
	\[
	\lim_{\eps\to 0}\iint_{Q} (V_{\eps}(x,y)-V(x,y))U(x,y)\ dxdy=0,
	\]
	This implies that $V_\eps\weakto V$ in $L^2(Q)$. Since $\|V_\eps\|_{L^2(Q)}^2\to \|V\|_{L^2(Q)}^2$ (by \eqref{l2}), this yields that $V_\eps\to V$ in $L^2(Q)$, and hence $\lim\limits_{\eps\to 0}\cE(v-v_{\eps},v-v_{\eps})=0.$

Finally, by Lemma~\ref{A2}, the fact that $\int_{\rn}(v_\eps-v)^2\ln(e+|x|^2)\, dx\to 0$ as $\eps\to 0$ follows from the theory of $L^2$-weighted spaces with weights in the $A_2$ Muckenhoupt class, see \cite[Theorem 2.1.4]{T00}.    
\end{proof}

\begin{proposition}\label{satz:dicht1}
$D^{log}(\rn)=\overline{C^{\infty}_c(\R^N)}^{\|\cdot\|}$.
\end{proposition}
\begin{proof}
	Let $v\in D^{log}(\rn)$ and let $\vartheta_n\in C^{0,1}_c(\R^N)$ for $n\in \N$ be given by Lemma~\ref{lemma:dicht-vorarbeit} such that $\|v-\vartheta_n v\|<\frac{1}{n}$. Then $v_n:=\vartheta_n v\in D^{log}(\rn)$ and there is $R_n>0$ with $v_n\equiv 0$ on $\R^N\setminus B_{R_n}(0)$. Next, let $(\rho_\eps)_{\eps\in(0,1]}$ be a Dirac sequence and denote $v_{n,\eps}:=\rho_\eps\ast v_n$. Then $v_{n,\eps}\in C^{\infty}_c(\R^N)$ for all $n\in \N$, $\eps\in(0,1]$ and $\|v-v_{n,\eps}\|\leq \|v-v_n\|+ \|v_n-v_{n,\eps}\|\leq \frac{1}{n}+ \|v_n-v_{n,\eps}\|.$ The claim now follows from Lemma~\ref{mollifier:lemma} letting $\eps\to 0$.
\end{proof}

We close this appendix with a density Lemma that is the key argument for Lemma~\ref{density:lem2}. Recall that $B_g(z,\rho)$ is the geodesic ball on $\S$ centered at $z\in \S$ of radius $\rho>0$.
\begin{lemma}\label{density:lem}
Let $u\in C^\infty(\S)$. There is $u_n\in C_c^\infty(\S\backslash \{-e_{N+1}\})$ such that $\lim_{n\to\infty}\|u-u_n\|_{\H}=0$.
\end{lemma}
\begin{proof}
 Let $u\in C^\infty(\S)$ and, for $n\in \N$, let $\rho_n\in C^\infty(\S)$ so that $0\leq \rho_n\leq 1$ and
\begin{align}\label{rho:p}
\rho_n=0\quad \text{ in }\S\backslash B_g(-e_{N+1},\tfrac{2}{n}),\qquad 
\rho_n=1\quad \text{ in }B_g(-e_{N+1},\tfrac{1}{n}),\qquad 
|\rho_n(z)-\rho_n(\zeta)|<Cn|z-\zeta|
\end{align}
for all $z,\zeta\in\S$ and for some $C>0.$ In the following, $C>0$ denotes possibly different constants independent of $n.$ Let $u_n:=(1-\rho_n) u.$ Then,
\begin{align}
\cE&(u_n-u,u_n-u)
=\int_{\S}\int_{\S}\frac{
|\rho_n(z)u(z)-\rho_n(\zeta)u(\zeta)|^2
}{|z-\zeta|^N}\, dV_g(z)\, dV_g(\zeta)\notag\\
&\leq 2\int_{\S}\int_{\S}\frac{
|\rho_n(z)|^2|u(z)-u(\zeta)|^2
+|u(\zeta)|^2|\rho_n(z)-\rho_n(\zeta)|^2
}{|z-\zeta|^N}\, dV_g(z)\, dV_g(\zeta)\notag\\
&\leq 2\int_{\S}|\rho_n(z)|^2\int_{\S}\frac{
|u(z)-u(\zeta)|^2}{|z-\zeta|^N}\, dV_g(\zeta)\, dV_g(z)
+\|u\|^2_{L^\infty(\S)}\int_{\S}\int_{\S}\frac{|\rho_n(z)-\rho_n(\zeta)|^2
}{|z-\zeta|^N}\, dV_g(z)\, dV_g(\zeta).\label{1}
\end{align}
Using that $|u(z)-u(\zeta)|<C|z-\zeta|$ for all $z,\zeta\in \S$,
\begin{align}
\int_{\S}|\rho_n(z)|^2\int_{\S}\frac{
|u(z)-u(\zeta)|^2}{|z-\zeta|^N}\, dV_g(\zeta)\, dV_g(z)
&\leq C\int_{\S}|\rho_n(z)|^2\int_{\S}|z-\zeta|^{2-N}\, dV_g(\zeta)\, dV_g(z)\notag\\
&\leq C|B_g(-e_{N+1},\tfrac{2}{n})|=C(\tfrac{2}{n})^N,\label{2}
\end{align}
where we used that $\eta(z):=\int_{\S}|z-\mu|^{2-N}\, dV_g(\mu)$ is uniformly bounded in $\S$, see \eqref{eta:bdd}.

We claim that 
\begin{align}\label{3}
\int_{\S}\int_{\S}\frac{|\rho_n(z)-\rho_n(\zeta)|^2
}{|z-\zeta|^N}\, dV_g(z)\, dV_g(\zeta)\to 0\qquad \text{ as }n\to \infty.
\end{align}
Let $\omega_1:=B_g(-e_{N+1},\tfrac{2}{n})$, $\omega_2:=\S\backslash B_g(-e_{N+1},\tfrac{2}{n}),$ 
and note that $\S = \omega_1 \cup \omega_2.$ Since,
\begin{align}\label{31}
\text{$\rho_n(z)-\rho_n(\zeta)=0$ on $\omega_2\times\omega_2,$ }    
\end{align}
we have, for $N\geq 3,$ by \eqref{eta:bdd}, that 
 \begin{align}
\iint_{\S\times \S}\frac{
|\rho_n(z)-\rho_n(\zeta)|^2
}{|z-\zeta|^N}\, dV_g(z)\, dV_g(\zeta)
&\leq Cn^2\int_{\omega_1}
\int_{\S}|z-\zeta|^{2-N}\, dV_g(z)\, dV_g(\zeta)\notag\\
&\leq Cn^2\|\eta\|_{L_g^\infty(\S)}|\omega_1|
=Cn^{2-N}\to 0\qquad \text{ as }n\to\infty.\label{33}
\end{align}

The cases $N=1,2$ are more delicate. By symmetry, it suffices to consider the sets $U:=\omega_1\times \omega_2$ and $V:=\omega_1\times \omega_1$. Then, using \eqref{rho:p} and that $2-N\geq 0$,
\begin{align}
\iint_{V}\frac{
|\rho_n(z)-\rho_n(\zeta)|^2
}{|z-\zeta|^N}\, dV_g(z)\, dV_g(\zeta)
&\leq Cn^2\int_{\omega_1}
\int_{\omega_1}|z-\zeta|^{2-N}\, dV_g(z)\, dV_g(\zeta)
\leq Cn^2(\tfrac{1}{n})^{2-N}|\omega_1|^2
=Cn^{-N}\to 0\label{32}
\end{align}
as $n\to \infty.$

Next, for $U$, recall the notation given at the beginning of Section \ref{prop:sec}. For $\varepsilon\in(0,r_0)$ fixed, let $n$ be large enough so that $\varepsilon>2/n$. As in Section \ref{prop:sec}, the inverse of the exponential map centered at the south pole gives a chart
\[
h:=\exp_{-e_{N+1}}^{-1}:B_g(-e_{N+1},\varepsilon)\rightarrow B_\varepsilon(0)\subset\mathbb{R}^N.
\]
As this chart gives normal coordinates around the south pole, it follows that
\begin{equation}\label{NormalCoordinatesEquality}
d_g(-e_{N+1},\zeta)=\vert h(\zeta)\vert\qquad \text{ for all }\;\zeta\in B_g(-e_{N+1},\varepsilon),
\end{equation}
and there exists $C>0$ such that
\begin{equation}\label{NormalCoordinatesInequality}
\frac{1}{C}\vert h(z) - h(\zeta) \vert \leq d_{g}(z,\zeta) \leq C \vert h(z) - h(\zeta)\vert\qquad \text{ for all }\;z,\zeta\in B_g(-e_{N+1},\varepsilon).
\end{equation}
Let $\omega_3:=B_g(-e_{N+1},\varepsilon)\backslash B_g(-e_{N+1},2/n)$, and notice that $U = \left(\omega_1\times \mathbb{S}^N\backslash B_g(-e_{N+1},\varepsilon)\right)\bigcup \left(\omega_1\times \omega_3\right),$ so it suffices to show that the integral in each of these two sets goes to zero as $n\rightarrow \infty$. Indeed, if $z\in\omega_1$ and if $\zeta\in \mathbb{S}^N\backslash B_g(-e_{N+1},\varepsilon)$, then 
$
\vert z - \zeta \vert
\geq c d_g(z,\zeta)\geq c(\varepsilon - 2/n)$ for some $c>0$, and, consequently,
\[
\vert z -\zeta\vert^{-N}\leq c^{-N}\left( \frac{n}{n\varepsilon-2} \right)^N.
\]
Thus, as $\vert \rho_n\vert\leq 1$,
\[
\int_{\omega_1}\int_{\mathbb{S}^N\backslash B_g(-e_{N+1},\varepsilon)} \frac{ |\rho_n(z)-\rho_n(\zeta)|^2
}{|z-\zeta|^N}\, dV_g(z)\, dV_g(\zeta) \leq C \left( \frac{n}{n\varepsilon-2} \right)^N \vert \omega_1\vert \vert\mathbb{S}^N\vert\leq C \left( \frac{n}{n\varepsilon-2} \right)^N \left(\frac{1}{n}\right)^N\rightarrow 0,
\]
as $n\rightarrow\infty$. 

Finally, we see that the integral defined in $\omega_1\times \omega_3$ also goes to zero as $n\rightarrow\infty$. Define the function $\Upsilon:\omega_1\rightarrow\mathbb{R}$ given by
\[
\Upsilon(z):= \int_{\omega_3}\vert z  -\zeta\vert^{-N}\;dV_g(\zeta).
\]
Then, for every $z\in\omega_1$, if $x=h(z)$, using \eqref{ComparisonGeodesicCordalDistance} together with \eqref{e1}, \eqref{NormalCoordinatesEquality}, and \eqref{NormalCoordinatesInequality}, it follows that

\begin{align*}
\Upsilon(z) &\leq C\int_{\omega_3} \vert h(z) - h(\zeta)\vert^{-N} \;dV_g(\zeta)
=C \int_{B_\varepsilon(0)\backslash B_{2/n}(0)} \vert x - y\vert^{-N} \sqrt{g}\circ h^{-1}(y)\;dy
\\
&\leq C \int_{B_\varepsilon(0)\backslash B_{2/n}(0)} \vert x - y\vert^{-N} \;dy= C \int_{B_\varepsilon(x)\backslash B_{2/n}(x)} \vert w\vert^{-N} \;dw.
\end{align*}
Since $
B_\varepsilon(x)\backslash B_{2/n}(x)\subset B_{2\varepsilon}(0)\backslash B_{\frac{2}{n}-\vert x\vert}(0),
$
we have, by radial integration, that
\[
\int_{B_\varepsilon(x)\backslash B_{2/n}(x)} \vert w\vert^{-N} \;dw \leq \int_{B_{2\varepsilon}(0)\backslash B_{2/n-\vert x\vert}(0)} \vert w\vert^{-N} \;dw = C\int_{2/n -\vert x\vert}^{2\varepsilon}t^{-1}dt \leq C - \ln(2/n - \vert x\vert).
\]
Hence, $\Upsilon(z) \leq C -  \ln(2/n - \vert h(z)\vert)$ for all $z\in \omega_1.$ With this estimate and using \eqref{NormalCoordinatesEquality}, since $\vert\rho_n\vert\leq 1,$ we have that
\begin{align*}
\iint_{\omega_1\times \omega_3}\frac{\vert \rho_n(z) - \rho_n(\zeta)\vert^2}{\vert z - \zeta\vert^N} \;dV_g(\zeta)\;dV_g(z)
&\leq C \int_{\omega_1}\Upsilon(z)\;dV_g(z)\leq C\vert \omega_1\vert -C\int_{\omega_1}\ln(2/n - \vert h(z)\vert)\; dV_g(z)\\
&\leq Cn^{-N} -C\int_{B_{2/n}(0)}\ln(2/n - \vert x\vert)\;dx 
\leq Cn^{-N}+C\int_0^{2/n}|\ln(|t|)|\;dt\rightarrow 0,
\end{align*}
 as  $n\rightarrow\infty,$ from which we conclude that  $\iint_{\omega_1\times \omega_3} \frac{\vert \rho_n(z) - \rho_n(\zeta)\vert^2}{\vert z - \zeta\vert^N} \;dV_g(\zeta)\;dV_g(z)\rightarrow 0,$ as $n\rightarrow\infty,$ and \eqref{3} follows in the cases $N=1,2$. The lemma now follows from \eqref{1}, \eqref{2}, and \eqref{3}.
\end{proof}

\section*{Acknowledgments}

We thank the anonymous referee for useful comments and suggestions that improved this manuscript.  We also thank Judith Campos Cordero for very helpful discussions. 

\section*{Declarations}

\subsection*{Ethical Approval}
This declaration is not applicable to the present work.

\subsection*{Funding}
J.C. Fernández is supported by CONAHCYT grant CBF2023-2024-116 (Mexico) and  by UNAM-DGAPA-PAPIIT grant IN110225 (Mexico). A. Saldaña is supported by CONAHCYT grant CBF2023-2024-116 (Mexico) and by UNAM-DGAPA-PAPIIT grant IN102925 (Mexico).

\subsection*{Availability of data and materials}
This declaration is not applicable to the present work.

\subsection*{AI Use}
All mathematical ideas, methods, and proof strategies were developed entirely by the authors, who also wrote the manuscript in full without AI assistance. Upon completion, the authors used ChatGPT-6 Astra solely for proofreading. Its suggestions were critically reviewed and those deemed appropriate were incorporated.

\bibliographystyle{siam}
\bibliography{refs}

@article{AoDelaTorreGonzalez2022,
  author  = {Ao, Weiwei and DelaTorre, Azahara and Gonz{\'a}lez, Mar{\'i}a del Mar},
  title   = {Symmetry and Symmetry Breaking for the Fractional {Caffarelli--Kohn--Nirenberg} Inequality},
  journal = {Journal of Functional Analysis},
  volume  = {282},
  number  = {11},
  pages   = {109438},
  year    = {2022},
  doi     = {10.1016/j.jfa.2022.109438}
}

@article{A97,
  title={On some inequalities for the gamma and psi functions},
  author={Alzer, Horst},
  journal={Mathematics of computations},
  volume={66},
  number={217},
  pages={373-389},
  year = {1997},
  publisher = {}
}

@incollection {DM18,
    AUTHOR = {Gonz\'alez, Mar\'ia del Mar},
     TITLE = {Recent progress on the fractional {L}aplacian in conformal
              geometry},
 BOOKTITLE = {Recent developments in nonlocal theory},
     PAGES = {236--273},
 PUBLISHER = {De Gruyter, Berlin},
      YEAR = {2018},
      ISBN = {978-3-11-057156-1; 978-3-11-057155-4},
   MRCLASS = {58J05 (35P25 35R11 35S05 53A30)},
  MRNUMBER = {3824214},
MRREVIEWER = {Matteo\ Cozzi},
       DOI = {10.1515/9783110571561-008},
       URL = {https://doi.org/10.1515/9783110571561-008},
}

@article{CFS25,
  title={Fractional {Q}-curvature on the sphere and optimal partitions},
  author={Chang-Lara, H{\'e}ctor A and Fern{\'a}ndez, Juan Carlos and Salda{\~n}a, Alberto},
  journal={Journal of the London Mathematical Society},
  volume={112},
  number={6},
  pages={e70366},
  year={2025},
  publisher={Wiley Online Library}
}

@article {CW19,
    AUTHOR = {Chen, Huyuan and Weth, Tobias},
     TITLE = {The {D}irichlet problem for the logarithmic {L}aplacian},
   JOURNAL = {Comm. Partial Differential Equations},
  FJOURNAL = {Communications in Partial Differential Equations},
    VOLUME = {44},
      YEAR = {2019},
    NUMBER = {11},
     PAGES = {1100--1139},
      ISSN = {0360-5302,1532-4133},
   MRCLASS = {35R11 (35B50 35B51 35D30)},
  MRNUMBER = {3995092},
MRREVIEWER = {Mehdi\ Nategh},
       DOI = {10.1080/03605302.2019.1611851},
       URL = {https://doi.org/10.1080/03605302.2019.1611851},
}

@article {CG2011,
    AUTHOR = {Chang, Sun-Yung Alice and Gonz\'alez, Mar\'ia del Mar},
     TITLE = {Fractional {L}aplacian in conformal geometry},
   JOURNAL = {Adv. Math.},
  FJOURNAL = {Advances in Mathematics},
    VOLUME = {226},
      YEAR = {2011},
    NUMBER = {2},
     PAGES = {1410--1432},
      ISSN = {0001-8708,1090-2082},
   MRCLASS = {58J60 (35R11)},
  MRNUMBER = {2737789},
MRREVIEWER = {Luca\ Lorenzi},
       DOI = {10.1016/j.aim.2010.07.016},
       URL = {https://doi.org/10.1016/j.aim.2010.07.016},
}

@article{CZ24,
  title={On positive solutions of critical semilinear equations involving the logarithmic {L}aplacian},
  author={Chen, Huyuan and Zhou, Feng},
  journal={arXiv preprint arXiv:2409.04797},
  year={2024}
}

@article{CDP18,
  title={Nonlocal operators of order near zero.},
  author={Correa, Ernesto and De Pablo, Arturo},
  journal={Journal of Mathematical Analysis and Applications},
  year={2018},
volume={461},
number ={1},
publisher = {Taylor \& Francis},
  pages={837
–
867},
doi={10.1016/j.jmaa.2017.12.011}
}

@article{FKT20,
  title={Classification of solutions of an equation related to a conformal log {S}obolev inequality},
  author={Frank, Rupert L and K{\"o}nig, Tobias and Tang, Hanli},
  journal={Advances in Mathematics},
  volume={375},
  pages={107395},
  year={2020},
  publisher={Elsevier}
}

@article{B92,
  title={{S}obolev inequalities, the {P}oisson semigroup, and analysis on the sphere {$S^n$}.},
  author={Beckner, William},
  journal={Proceedings of the National Academy of Sciences},
  volume={89},
  number={11},
  pages={4816--4819},
  year={1992}
}

@article{B97,
  title={Logarithmic {S}obolev inequalities and the existence of singular integrals},
  author={Beckner, William},
  journal={Forum Math.},
  volume={9},
  number={3},
  pages={303--323},
  year={1997}
}

@article {GJMS1992,
    AUTHOR = {Graham, C. R. and Jenne, R. and Mason, L. J. and
              Sparling, G. A. J.},
     TITLE = {Conformally invariant powers of the {L}aplacian. {I}.
              {E}xistence},
   JOURNAL = {J. London Math. Soc. (2)},
  FJOURNAL = {Journal of the London Mathematical Society. Second Series},
    VOLUME = {46},
      YEAR = {1992},
    NUMBER = {3},
     PAGES = {557--565},
      ISSN = {0024-6107},
   MRCLASS = {58G35 (53A30)},
  MRNUMBER = {1190438},
MRREVIEWER = {Michael G. Eastwood},
       DOI = {10.1112/jlms/s2-46.3.557},
       URL = {https://doi.org/10.1112/jlms/s2-46.3.557},
}

@article {GZ2003,
    AUTHOR = {Graham, C. Robin and Zworski, Maciej},
     TITLE = {Scattering matrix in conformal geometry},
   JOURNAL = {Invent. Math.},
  FJOURNAL = {Inventiones Mathematicae},
    VOLUME = {152},
      YEAR = {2003},
    NUMBER = {1},
     PAGES = {89--118},
      ISSN = {0020-9910,1432-1297},
   MRCLASS = {58J50},
  MRNUMBER = {1965361},
MRREVIEWER = {Andrew\ W.\ Hassell},
       DOI = {10.1007/s00222-002-0268-1},
       URL = {https://doi.org/10.1007/s00222-002-0268-1},
}

@article{HS22,
  title={Small order asymptotics for nonlinear fractional problems},
  author={Hern{\'a}ndez Santamar{\'\i}a, V{\'\i}ctor and Saldaña, Alberto},
  journal={Calculus of Variations and Partial Differential Equations},
  volume={61},
  number={3},
  pages={92},
  year={2022},
  publisher={Springer}
}

@article{HLS25,
  title={Optimal boundary regularity and a {H}opf-type lemma for {D}irichlet problems involving the logarithmic {L}aplacian},
  author={Hern{\'a}ndez-Santamar{\'\i}a, V{\'\i}ctor and R{\'\i}os, Luis Fernando L{\'o}pez and Salda{\~n}a, Alberto},
  journal={Discrete and Continuous Dynamical Systems},
  volume={45},
  number={1},
  pages={1--36},
  year={2025},
  publisher={Discrete and Continuous Dynamical Systems}
}

@article{GQ13,
  title={Fractional conformal {L}aplacians and fractional {Y}amabe problems},
  author={Gonz{\'a}lez, Mar{\'\i}a del Mar and Qing, Jie},
  journal={Analysis \& PDE},
  volume={6},
  number={7},
  pages={1535--1576},
  year={2013},
  publisher={Mathematical Sciences Publishers}
}

@article{FS08,
  title={Hardy-{L}ieb-{T}hirring inequalities for fractional {S}chr{\"o}dinger operators},
  author={Frank, Rupert and Lieb, Elliott and Seiringer, Robert},
  journal={Journal of the American Mathematical Society},
  volume={21},
  number={4},
  pages={925--950},
  year={2008}
}

@article{B95,
  title={Pitt’s inequality and the uncertainty principle},
  author={Beckner, William},
  journal={Proceedings of the American Mathematical Society},
  volume={123},
  number={6},
  pages={1897--1905},
  year={1995}
}

@article{Y99,
  title={Sharp constants in the {H}ardy--{R}ellich inequalities},
  author={Yafaev, Dimitri},
  journal={Journal of Functional Analysis},
  volume={168},
  number={1},
  pages={121--144},
  year={1999},
  publisher={Academic Press}
}

@article{H77,
  title={Spectral theory of the operator $(p^2+ m^2)^{1/2}-{Z}e^2/r$},
  author={Herbst, Ira W},
  journal={Communications in Mathematical Physics},
  volume={53},
  number={3},
  pages={285--294},
  year={1977},
  publisher={Springer}
}

@article{BGV21,
  title={Characterisation of homogeneous fractional {S}obolev spaces},
  author={Brasco, Lorenzo and G{\'o}mez-Castro, David and V{\'a}zquez, Juan Luis},
  journal={Calculus of Variations and Partial Differential Equations},
  volume={60},
  number={2},
  pages={60},
  year={2021},
  publisher={Springer}
}

@article{JSW20,
  title={A new look at the fractional {P}oisson problem via the logarithmic {L}aplacian},
  author={Jarohs, Sven and Saldaña, Alberto and Weth, Tobias},
  journal={Journal of Functional Analysis},
  volume={279},
  number={11},
  pages={108732},
  year={2020},
  publisher={Elsevier}
}

@article{CS24,
  title={Classical solutions to integral equations with zero order kernels},
  author={Chang-Lara, H{\'e}ctor A and Saldaña, Alberto},
  journal={Mathematische Annalen},
  volume={389},
  number={2},
  pages={1463--1515},
  year={2024},
  publisher={Springer}
}

@article{AS23,
  title={Small order limit of fractional {D}irichlet sublinear-type problems},
  author={Angeles, Felipe and Salda{\~n}a, Alberto},
  journal={Fractional Calculus and Applied Analysis},
  volume={26},
  number={4},
  pages={1594--1631},
  year={2023},
  publisher={Springer}
}

@article{JWS24,
  title={Differentiability of the nonlocal-to-local transition in fractional {P}oisson problems},
  author={Jarohs, Sven and Salda{\~n}a, Alberto and Weth, Tobias},
  journal={Potential Analysis},
  pages={1--23},
  year={2024},
  publisher={Springer}
}

@article{HJSS25,
  title={{FEM} for 1{D}-problems involving the logarithmic {L}aplacian: error estimates and numerical implementation},
  author={Hern{\'a}ndez-Santamar{\'\i}a, V{\'\i}ctor and Jarohs, Sven and Salda{\~n}a, Alberto and Sinsch, Leonard},
  journal={Computers \& Mathematics with Applications},
  volume={192},
  pages={189--211},
  year={2025},
  publisher={Elsevier}
}

@article{LW21,
  title={Spectral properties of the logarithmic {L}aplacian},
  author={Laptev, Ari and Weth, Tobias},
  journal={Analysis and Mathematical Physics},
  volume={11},
  number={3},
  pages={133},
  year={2021},
  publisher={Springer}
}

@article{CV24,
  title={The {C}auchy problem associated to the logarithmic {L}aplacian with an application to the fundamental solution},
  author={Chen, Huyuan and V{\'e}ron, Laurent},
  journal={Journal of Functional Analysis},
  volume={287},
  number={3},
  pages={110470},
  year={2024},
  publisher={Elsevier}
}

@article{FJW22,
  title={Small order asymptotics of the {D}irichlet eigenvalue problem for the fractional {L}aplacian},
  author={Feulefack, Pierre Aime and Jarohs, Sven and Weth, Tobias},
  journal={Journal of Fourier Analysis and Applications},
  volume={28},
  number={2},
  pages={18},
  year={2022},
  publisher={Springer}
}

@article{chw23,
  title={An extension problem for the logarithmic {L}aplacian},
  author={Chen, Huyuan and Hauer, Daniel and Weth, Tobias},
  journal={arXiv preprint arXiv:2312.15689},
  year={2023}
}

@article{FJ23,
  title={Nonlocal operators of small order},
  author={Feulefack, Pierre Aime and Jarohs, Sven},
  journal={Annali di Matematica Pura ed Applicata (1923-)},
  volume={202},
  number={4},
  pages={1501--1529},
  year={2023},
  publisher={Springer}
}

@article{CV23,
  title={Bounds for eigenvalues of the {D}irichlet problem for the logarithmic {L}aplacian},
  author={Chen, Huyuan and V{\'e}ron, Laurent},
  journal={Advances in Calculus of Variations},
  volume={16},
  number={3},
  pages={541--558},
  year={2023},
  publisher={De Gruyter}
}

@article{djf24,
  title={The {D}irichlet Problem For the Logarithmic p-{L}aplacian},
  author={Dyda, Bart{\l}omiej and Jarohs, Sven and Sk, Firoj},
  journal={arXiv preprint arXiv:2411.11181},
  year={2024}
}

@article{chen2025logarithmic,
  title={Logarithmic {L}aplacian on General {R}iemannian Manifolds},
  author={Chen, Rui},
  journal={arXiv preprint arXiv:2506.19311},
  year={2025}
}

@article{C25,
  title={On m-order logarithmic {L}aplacians and the applications},
  author={Chen, Huyuan},
  journal={Analysis and Applications},
  pages={1--43},
  year={2025},
  publisher={World Scientific}
}

@article{CSS21,
  title={Phase separation, optimal partitions, and nodal solutions to the {Y}amabe equation on the sphere},
  author={Clapp, M{\'o}nica and Salda{\~n}a, Alberto and Szulkin, Andrzej},
  journal={International Mathematics Research Notices},
  volume={2021},
  number={5},
  pages={3633--3652},
  year={2021},
  publisher={Oxford University Press}
}

@article{CFS21,
  title={Critical polyharmonic systems and optimal partitions.},
  author={Clapp, M{\'o}nica and Fern{\'a}ndez, Juan Carlos and Salda{\~n}a, Alberto},
  journal={Communications on Pure \& Applied Analysis},
  volume={20},
  number={11},
  year={2021}
}

@article{D86,
  title={On a conformally invariant elliptic equation on ${R}^n$},
  author={Weiyue, Ding},
  journal={Commun. Math. Phys},
  volume={107},
  pages={331--335},
  year={1986}
}

@book{T00,
title = {Nonlinear Potential Theory and Weighted Sobolev Spaces},
publisher = {Springer Berlin, Heidelberg},
address = {},
pages = {XII,180},
series = {Lecture Notes in Mathematics},
year = {2000},
isbn = {},
doi = {https://doi.org/10.1007/BFb0103908},
author = {Bengt Ove Turesson},
}

@article{AGV25,
  title={The {B}rezis-{N}irenberg and logistic problem for the Logarithmic {L}aplacian},
  author={Arora, Rakesh and Giacomoni, Jacques and Vaishnavi, Arshi},
  journal={arXiv preprint arXiv:2504.18907},
  year={2025}
}

@article{PS25,
  title={Antisymmetric maximum principles and {H}opf’s lemmas for the Logarithmic {L}aplacian, with applications to symmetry results},
  author={Pollastro, Luigi and Soave, Nicola},
  journal={Annali di Matematica Pura ed Applicata (1923-)},
  pages={1--19},
  year={2025},
  publisher={Springer}
}

@article{MayerNdiaye2024,
  author  = {Mayer, Martin and Ndiaye, Cheikh Birahim},
  title   = {Fractional {Y}amabe Problem on Locally Flat Conformal Infinities of {P}oincar{\'e}--{E}instein Manifolds},
  journal = {International Mathematics Research Notices},
  volume  = {2024},
  number  = {3},
  pages   = {2561--2621},
  year    = {2024},
  doi     = {10.1093/imrn/rnad195}
}

@article{BettiolGonzalezMaalaoui2024,
  author  = {Bettiol, Renato G. and Gonz{\'a}lez, Mar{\'i}a del Mar and Maalaoui, Ali},
  title   = {Multiplicity of Singular Solutions to the Fractional Yamabe Problem on Spheres},
  journal = {Journal of Differential Equations},
  volume  = {389},
  pages   = {285--304},
  year    = {2024},
  doi     = {10.1016/j.jde.2024.01.019}
}

@article{GurskyMalchiodi2015,
  author = {Gursky, M. J. and Malchiodi, A.},
  title = {A strong maximum principle for the {Paneitz} operator and a non-local flow for the {$Q$}-curvature},
  journal = {Journal of the European Mathematical Society},
  fjournal = {Journal of the European Mathematical Society (JEMS)},
  volume = {17},
  number = {9},
  pages = {2137--2173},
  year = {2015},
  doi = {10.4171/JEMS/553},
  url = {https://doi.org/10.4171/JEMS/553},
  issn = {1435-9855},
  mrclass = {53C44 (35B50 35K58 35R01 53A30)},
  mrnumber = {3420504},
  mrreviewer = {Gang Li},
}

@article{HangYang2016,
  author = {Hang, F. and Yang, P. C.},
  title = {{$Q$}-curvature on a class of manifolds with dimension at least 5},
  journal = {Communications on Pure and Applied Mathematics},
  volume = {69},
  pages = {1452--1491},
  year = {2016},
}

@misc{MazumdarNdiaye2026,
  author        = {Mazumdar, Saikat and Ndiaye, Cheikh Birahim},
  title         = {Barycenter technique for the higher order {$Q$}-curvature equation},
  year          = {2026},
  eprint        = {2603.06249},
  archivePrefix = {arXiv},
  primaryClass  = {math.DG},
  url           = {https://arxiv.org/abs/2603.06249}
}

@misc{MazumdarPremoselli2025,
  author        = {Mazumdar, Saikat and Premoselli, Bruno},
  title         = {Compactness of conformal metrics with constant {$Q$}-curvature of higher order},
  year          = {2025},
  eprint        = {2510.00888},
  archivePrefix = {arXiv},
  primaryClass  = {math.AP},
  url           = {https://arxiv.org/abs/2510.00888}
}

@article{MazumdarVetois2024,
  author  = {Mazumdar, Saikat and V{\'e}tois, J{\'e}r{\^o}me},
  title   = {Existence results for the higher-order {$Q$}-curvature equation},
  journal = {Calculus of Variations and Partial Differential Equations},
  volume  = {63},
  number  = {6},
  pages   = {151},
  year    = {2024},
  doi     = {10.1007/s00526-024-02757-x}
}

\noindent
\begin{minipage}[t]{0.48\textwidth}
\raggedright
\textbf{Juan Carlos Fernández}\\
Departamento de Matemáticas, Facultad de Ciencias\\
Universidad Nacional Autónoma de México\\
Circuito Exterior, Ciudad Universitaria\\
04510 Coyoacán, Ciudad de México, Mexico.\\
\texttt{jcfmor@ciencias.unam.mx}
\end{minipage}
\hfill
\begin{minipage}[t]{0.48\textwidth}
\raggedright
\textbf{Alberto Saldaña}\\
Instituto de Matemáticas\\
Universidad Nacional Autónoma de México\\
Circuito Exterior, Ciudad Universitaria\\
04510 Coyoacán, Ciudad de México, Mexico.\\
\texttt{alberto.saldana@im.unam.mx}
\end{minipage}

\end{document}